\documentclass{amsart}
\usepackage{amsmath,amsthm,amssymb}
\usepackage{marvosym}
\usepackage{graphicx}
\usepackage{color}
\usepackage{epsfig}
\usepackage{morefloats}
\usepackage{fullpage}
\usepackage{MnSymbol}
\usepackage[active]{srcltx}

\theoremstyle{plain}
\newtheorem{thm}{Theorem}[section]
\newtheorem{cor}[thm]{Corollary}
\newtheorem{conj}[thm]{Conjecture}
\newtheorem{prop}[thm]{Proposition}
\newtheorem{lemma}[thm]{Lemma}
\newtheorem{claim}[thm]{Claim}

\newtheorem{question}[thm]{Question}
\newtheorem*{claim*}{Claim}

\theoremstyle{definition}

\newtheorem{remark}[thm]{Remark}

\newcommand{\comment}[1]{}

\newcommand{\bdry}{\ensuremath{\partial}}

\DeclareMathOperator{\Id}{Id}
\DeclareMathOperator{\vol}{vol}

\newcommand{\Q}{\ensuremath{\mathbb{Q}}}

\newcommand{\Z}{\ensuremath{\mathbb{Z}}}

\newcommand{\mobius}{M\"{o}bius }

\newcommand{\bgi}{{\mbox{\sc bgi}}}
\newcommand{\bgii}{{\mbox{\sc bgii}}}
\newcommand{\bgiii}{{\mbox{\sc bgiii}}}
\newcommand{\bgiv}{{\mbox{\sc bgiv}}}
\newcommand{\bgv}{{\mbox{\sc bgv}}}
\newcommand{\gofk}{{\mbox{\sc gofk}}}
\newcommand{\spor}{{\mbox{\sc spor}}}

\newcommand{\executeiffilenewer}[3]{%
\ifnum\pdfstrcmp{\pdffilemoddate{#1}}%
{\pdffilemoddate{#2}}>0%
{\immediate\write18{#3}}\fi%
}

\title{
Some knots in $S^1 \times S^2$ with lens space surgeries.}

\author{Kenneth L.\ Baker}
\address{
Department of Mathematics
University of Miami,
PO Box 249085
Coral Gables, FL 33124-4250}
\email{k.baker@math.miami.edu}
\urladdr{http://math.miami.edu/\char126 kenken}

\author{Dorothy Buck}
\address{Dept of Mathematics, Imperial College London, South Kensington, London England SW7 2AZ}
\email{d.buck@imperial.ac.uk}
\urladdr{http://www2.imperial.ac.uk/~dbuck/}

\author{Ana G. Lecuona}
\address{LATP, Aix-Marseille Universit\'e, Marseille, France}
\email{ana.lecuona@latp.univ-mrs.fr}
\urladdr{http://www.math.psu.edu/lecuona/}

\begin{document}

\begin{abstract}
We propose a classification of knots in $S^1 \times S^2$ that admit a longitudinal surgery to a lens space.  Any lens space obtainable by longitudinal surgery on some knots in $S^1 \times S^2$ may be obtained from a Berge-Gabai knot in a Heegaard solid torus of $S^1 \times S^2$, as observed by Rasmussen.  We show that there are yet two other families of knots: those that lie on the fiber of a genus one fibered knot and the `sporadic' knots.  All these knots in $S^1 \times S^2$ are both doubly primitive and spherical braids.   

This classification arose from generalizing Berge's list of doubly primitive knots in $S^3$, though we also examine how one might develop it using Lisca's embeddings of the intersection lattices of rational homology balls bounded by lens spaces as a guide. We conjecture that our knots constitute a complete list of doubly primitive knots in $S^1 \times S^2$ and reduce this conjecture to classifying the homology classes of knots in lens spaces admitting a longitudinal $S^1 \times S^2$ surgery.
\end{abstract}


\maketitle

\section{Introduction}
A knot $K$ in a $3$--manifold $M$ is {\em doubly primitive} if it may be embedded in a genus $2$ Heegaard surface of $M$ so that it represents a generator of each handlebody, i.e.\ in each handlebody there is a compressing disk that $K$ transversally intersects exactly once.  With such a doubly primitive presentation, surgery on $K$ along the slope induced by the Heegaard surface yields a lens space.  Berge introduced this concept of doubly primitive and provided twelve families (which partition into three broader families) of knots in $S^3$ that are doubly primitive \cite{bergeII}. The Berge Conjecture asserts that if longitudinal surgery on a knot in $S^3$ produces a lens space, then that knot admits a presentation as a doubly primitive knot in a genus $2$ Heegaard surface in $S^3$ in which the slope induced by the Heegaard surface is the surgery slope.  This conjecture is regarded as implicit in \cite{bergeII}.

This conjecture has a prehistory fueled by the classification of lens space surgeries on torus knots \cite{moser}, notable examples of longitudinal lens space surgeries on non-torus knots \cite{baileyrolfsen,therons}, the Cyclic Surgery Theorem \cite{cgls}, the resolution of the Knot Complement problem \cite{gl:kadbtc},  several treatments of lens space surgeries on satellite knots \cite{wu, bleilerlitherland, wang}, and the classification of surgeries on knots in solid tori producing solid tori \cite{gabaiI, gabaiII, bergeS1xD2} to name a few.  The modern techniques of Heegaard Floer homology \cite{OSfoundations1,OSfoundations2} opened new approaches that reinvigorated the community's interest and gave way to deeper insights of positivity \cite{hedden-positive}, fiberedness \cite{ni}, and simplicity \cite{OSLspace,HeddenSimple,RasmussenSimple}.  

One remarkable turn is Greene's solution to the Lens Space Realization Problem \cite{greene}.  Utilizing the correction terms of Heegaard Floer homology \cite{OScorrectionterms}, Greene adapts and enhances Lisca's lattice embedding ideas \cite{lisca} to determine not only  which lens spaces may be obtained by surgery on a knot in $S^3$ but also the homology classes of the corresponding dual knots in those lens spaces.  This gives the pleasant corollary that Berge's twelve families of doubly primitive knots in $S^3$ is complete.

Our present interest lies in the results of Lisca's work \cite{lisca} which, with an observation by Rasmussen \cite[Section 1.5]{greene}, solves the $S^1 \times S^2$ version of the Lens Space Realization Problem.  That is, the lens spaces which bound rational homology $4$--balls as determined by Lisca may each be obtained by longitudinal surgery on some knot in $S^1 \times S^2$.  (Note that if a lens space results from longitudinal surgery on a knot in $S^1 \times S^2$ then it necessarily bounds a rational homology $4$--ball.)  In fact, as Rasmussen observed, the standard embeddings into $S^1 \times S^2$ of the Berge-Gabai knots in solid tori with longitudinal surgeries yielding solid tori suffice. Due to the uniqueness of lattice embeddings in Lisca's situation versus the flexibility of lattice embeddings in his situation, Greene had initially conjectured that these accounted for all knots in $S^1 \times S^2$ with lens space surgeries \cite{greene}.    In this article we show that there are yet two more families of knots and probe their relationships with Lisca's lattice embeddings.  Indeed, here we begin a program to bring the status of the classification of knots in $S^1 \times S^2$ with lens space surgeries in line with the present state of the Berge Conjecture.  The main purpose of this article is to propose such a classification of knots and provide a foundation for showing our knots constitute all the doubly primitive knots in $S^1 \times S^2$.

\begin{conj}[Cf. Conjecture~1.9 \cite{greene}]\label{conj:main} 
 The knots in $S^1 \times S^2$ with a longitudinal surgery producing a lens space are all doubly-primitive. 
\end{conj}

\begin{conj}\label{conj:dp}
A doubly-primitive knot in $S^1 \times S^2$ is either a Berge-Gabai knot, a knot that embeds in the fiber of a genus one fibered knot, or a sporadic knot.
\end{conj}

The three families of knots in Conjecture~\ref{conj:dp} are analogous to the three broad of families of Berge's doubly primitive knots in $S^3$ and will be described below.

Section~\ref{sec:notation} contains some of the basic terminology and notation that will be used.

\subsection{Lens spaces obtained by surgery on knots in $S^1 \times S^2$} 
Lisca determines whether a $3$--dimensional lens space bounds a $4$--dimensional rational homology ball by studying the embeddings into the standard diagonal intersection lattice of the intersection lattice of the canonical plumbing manifold bounding that lens space, \cite{lisca}. From this and that lens spaces are the double branched covers of two-bridge links, Lisca obtains a classification of which two-bridge {\em knots} in $S^3 =\bdry B^4$ bound smooth disks in $B^4$ (i.e.\ are slice) and which two-component two-bridge links bound a smooth disjoint union of a disk and a M\"obius band  in $B^4$.  As part of doing so, he demonstrates that in the projection to $S^3$ these surfaces may be taken to have only ribbon singularities.   Indeed he shows this by using a single banding to transform these two-bridge links into the unlink, except for two families: one  for which he uses two bandings and another which was overlooked.

Via double branched covers and the Montesinos Trick, the operation of a banding lifts to the operation of a longitudinal surgery on a knot in the double branched cover of the original link.  Since the double branched cover of the two component unlink is $S^1 \times S^2$, Lisca's work shows in many cases that the lens spaces bounding rational homology balls contain a knot on which longitudinal surgery produces $S^1 \times S^2$.  In fact, the lens spaces bounding rational homology balls  are precisely those that contain a knot on which longitudinal surgery produces $S^1 \times S^2$:  Greene notes Rasmussen had observed that Lisca's list of lens spaces corresponds to those that may be obtained from considering the Berge-Gabai knots in solid tori with a solid torus surgery \cite{bergeS1xD2,gabaiI} as residing in a Heegaard solid torus of $S^1 \times S^2$, \cite[Section 1.5]{greene}.   By appealing to the classification of lens spaces up to homeomorphisms we may condense the statement as follows:

\begin{thm}[Rasmussen via {\cite[Section 1.5]{greene}}]
\label{thm:main}
The lens space $L$ may be transformed into $S^1 \times S^2$ by longitudinal surgery on a knot if and only if there are integers $m$ and $d$ such that $L$ is homeomorphic to one of the four lens spaces:
\begin{enumerate}
\item $L(m^2, md+1)$ such that $\gcd(m,d)=1$;
\item $L(m^2, md+1)$ such that $\gcd(m,d)=2$;
\item $L(m^2, d(m-1))$ such that $d$ is odd and divides $m-1$; or
\item $L(m^2, d(m-1))$ such that $d$ divides $2m+1$.
\end{enumerate}
\end{thm}
Note that we do permit $m$ and $d$ to be negative integers.
We will augment this theorem in Theorem~\ref{thm:mainhomology} with the homology classes known to contain the knots dual to these longitudinal surgeries from $S^1 \times S^2$.

Since the Berge-Gabai knots in solid tori all have tunnel number one, the corresponding knots in $S^1 \times S^2$ are strongly invertible.  Quotienting by this strong inversion gives the analogous result for bandings of two-bridge links to the unlink (of two components).
\begin{cor}
The two-bridge link $K$ may be transformed into the unlink by a single banding if and only if there are integers $m$ and $d$ such that $K$ is homeomorphic to one of the four two-bridge links:
\begin{enumerate}
\item $K(m^2, md+1)$ such that $\gcd(m,d)=1$;
\item $K(m^2, md+1)$ such that $\gcd(m,d)=2$;
\item $K(m^2, d(m-1))$ such that $d$ is odd and divides $m-1$; or
\item $K(m^2, d(m-1))$ such that $d$ divides $2m+1$,
\end{enumerate}
\end{cor}

\begin{remark}\label{rem:missing}
By an oversight in the statement of \cite[Lemma,7.2]{lisca}, a family of strings of integers was left out though they are produced by the proof (cf.\ \cite[Footnote p.\ 247]{Le}).  The use of this lemma in \cite[Lemma 9.3]{lisca} causes the second family in Theorem~\ref{thm:main} above to be missing from \cite[Definition~1.1]{lisca}.  
Consequentially, the associated  two-bridge links (these  necessarily have two components) were also not shown to bound a disjoint union of a disk and a M\"obius band in $B^4$ in that article.   (Also, we have swapped the order of the last two families.)
\end{remark}

Prompted by the uniqueness of Lisca's lattice embeddings (Lemma~\ref{lem:uniqueembedding}) and seemingly justified by
 Rasmussen's observation, Greene had originally conjectured that if a knot in $S^1 \times S^2$ admits a longitudinal lens space surgery, then it arises from a Berge-Gabai knot in a Heegaard solid torus of $S^1 \times S^2$, \cite[Conjecture~1.8]{greene}.  These knots belong to five families which we call {\sc bgi}, {\sc bgii}, {\sc bgiii}, {\sc bgiv}, and {\sc bgv} and refer to collectively as the family {\sc bg} of {\em Berge-Gabai knots}.  We show Greene's original conjecture is false by exhibiting two new families, {\sc gofk} and {\sc spor}, of knots in $S^1 \times S^2$ admitting longitudinal lens space surgeries. (In fact it turns out that Yamada had previously observed the {\sc gofk} family of knots admit lens space surgeries \cite{yamada}.)    Conjecture~\ref{conj:main} accordingly updates that conjecture with these two new families of doubly primitive knots in $S^1 \times S^2$. Conjecture~\ref{conj:dp} claims that there are no other doubly primitive knots.

\begin{thm}\label{thm:conj}  
The two families {\sc gofk} and {\sc spor} of knots in $S^1 \times S^2$ that admit a longitudinal lens space surgery, generically do not arise from Berge-Gabai knots.
\end{thm}

\begin{proof}
Lemma~\ref{lem:gofk} shows that generically the knots {\sc gofk} are hyperbolic and ``most'' have volume greater than the hyperbolic Berge-Gabai knots.  Lemma~\ref{lem:spor} shows that, except in two cases, regardless of choice of orientations, the lens space surgery duals to the knots {\sc spor} are not in the same homology class as the dual to any Berge-Gabai knot. 
\end{proof}

We provide explicit demonstrations of Theorem~\ref{thm:main} from two different perspectives, both of which produce the new families {\sc gofk} and {\sc spor} in addition to the Berge-Gabai knots.

Taking lead from Berge's list of doubly primitive knots in $S^3$ \cite{bergeII} and the descriptions of their associated tangles \cite{bakerI,bakerII} we first obtain tangle descriptions of the Berge-Gabai knots in $S^1 \times S^2$ (by way of tangle descriptions of the Berge-Gabai knots in solid tori in \cite{bakerbuck}) to provide one proof of Theorem~\ref{thm:main}.  Then we generate families {\sc gofk} and {\sc spor} analogous to Berge's families VII, VIII and IX, X, XI, XII respectively from which Theorem~\ref{thm:conj} falls.  

Alternatively, the embeddings (given by Lisca) of the intersection lattice of the negative definite plumbing manifold bounded by a lens space suggest where an initial blow-down ought occur to cause the entire plumbing diagram to collapse to a zero framed unknot, see section~\ref{sec:embeddings}.  It turns out that the perhaps more obvious ones happen to correspond to the duals to the {\sc bg} knots though the duals to the {\sc gofk} also arise, Lemma~\ref{lem:BGembedding}.  Having found the family {\sc spor} by the above tangle method, 
 we are able to identify blow-downs corresponding to the duals of these knots as well, though they do not usually cause the plumbing diagram to collapse completely.  We discuss this in section~\ref{sec:latticeknots}.

\subsection{Simple knots}

A {\em $(1,1)$--knot} is a knot $K$ that admits a presentation as a $1$--bridge knot with respect to a genus $1$ Heegaard splitting of the manifold $M$ that contains it.  That is, $M$ may be presented as the union of two solid tori $V_\alpha$ and $V_\beta$ in which each $K \cap V_\alpha$ and $K \cap V_\beta$ is a boundary parallel arc in the respective solid torus.   We say $K$ is {\em simple} if furthermore there are meridional disks of $V_\alpha$ and $V_\beta$ whose boundaries intersect minimally in the common torus $V_\alpha \cap V_\beta$ in $M$ such that each arc $K \cap V_\alpha$ and $K \cap V_\beta$ is disjoint from these meridional disks.  One may show there is a unique (oriented) simple knot in each (torsion) homology class of a lens space.   Let us write $K(p,q,k)$ for the simple knot in $M=L(p,q)$ oriented so that it represents the homology class $k \mu$ where (for a choice of orientation) $\mu$ is the homology class of the core curve of one of the Heegaard solid tori and $q\mu$ is the homology class of the other. Observe that  trivial knots are simple knots and, as such, permits both $S^3$ and $S^1 \times S^2$ to have a simple knot. There are no simple knots representing the non-torsion homology classes of $S^1 \times S^2$.  Non-trivial simple knots have also been called {\em grid number one} knots, e.g.\ in \cite{bakergrigsbyhedden, bakergrigsby} among others.

The Homma-Ochiai-Takahashi recognition algorithm for $S^3$ among genus $2$ Heegaard diagrams \cite{hot} says that a genus $2$ Heegaard diagram of $S^3$ is either the standard one or contains what is called a {\em wave}, see Section~\ref{sec:simplewave}.  A wave in a Heegaard diagram indicates the existence of a handle slide that will produce a new Heegaard diagram for the same manifold with fewer crossings. As employed by Berge \cite{bergeII}, the existence of waves ultimately tells us that any $(1,1)$--knot in a lens space with a longitudinal $S^3$ surgery is isotopic to a simple knot.
As the dual to a doubly primitive knot is necessarily a $(1,1)$--knot, it follows that the dual to a doubly primitive knot in $S^3$ is a simple knot in the resulting lens space.

Negami-Okita's study of reductions of diagrams of $3$--bridge links gives insights to the existence of wave moves on genus $2$ Heegaard diagrams.

\begin{thm}[Negami-Okita \cite{negamiokita}]\label{thm:negamiokita}
Every Heegaard diagram of genus $2$ for $S^1 \times S^2 \# L(p,q)$ may be transformed into one of the standard ones by a finite sequence of wave moves.
\end{thm}
Here, a {\em standard} genus $2$ Heegaard diagram $H=(\Sigma, \{\alpha_1,\alpha_2\},\{\beta_1, \beta_2\})$ for $S^1 \times S^2 \# L(p,q)$ is one for which  $\alpha_1$ and $\beta_1$ are parallel and disjoint from $\alpha_2 \cup \beta_2$,  and $\alpha_2 \cap \beta_2$ consists of exactly $p$ points.  If $p \neq1$, then the standard diagrams are not unique.  For our case at hand however, $p=1$ and the standard diagram is unique (up to homeomorphism).  This enables a proof of a result analogous to Berge's.

\begin{thm}\label{thm:simplewave}\
\begin{enumerate}
\item A $(1,1)$--knot in a lens space with a longitudinal $S^1 \times S^2$ surgery is a simple knot.
\item The dual to a doubly primitive knot in $S^1 \times S^2$ is a simple knot in the corresponding lens space.
\end{enumerate}
\end{thm}

A proof of this follows similarly to Berge's proof for doubly primitive knots in $S^3$ and their duals, though there is a technical issue one ought mind.  We will highlight this as we sketch the argument of a more general result in section~\ref{sec:simplewave} following Saito's treatment of Berge's work in the appendix of \cite{saito}.

\subsection{The known knots in lens spaces with longitudinal $S^1 \times S^2$ surgeries.}

Since our knots in families {\sc bg}, {\sc gofk}, and {\sc spor} are all doubly primitive, then by Theorem~\ref{thm:simplewave}  their lens space surgery duals are simple knots.  In particular, this means these duals 
are all at most $1$--bridge with respect to the Heegaard torus of the lens space, and thus they admit a nice presentation in terms of linear chain link surgery descriptions of the lens space.  
This surgery description (which we first obtained by simplifying ones suggested by the lattice embeddings)
 facilitates the calculation of the homology classes of these dual knots and hence their descriptions as simple knots.

Given the lens space $L(p,q)$, let $\mu$ and $\mu'$ be homology classes of the core curves of the Heegaard solid tori oriented so that $\mu' = q \mu$.  The homology class of a knot in $L(p,q)$ is given as its multiple of $\mu$.

\begin{thm}\label{thm:mainhomology}
The lens spaces $L(m^2,q)$ of Theorem~\ref{thm:main} may be obtained by longitudinal surgeries on the following simple knots $K=K(m^2,q,k)$  listed below.
\begin{enumerate}
\item $q=md+1$  such that $\gcd(m,d)=1$ and either 
\begin{itemize}
\item $k=\pm m$ so that $K$ is the dual to a {\sc bgi} knot or 
\item $k = \pm dm$ so that $K$ is the dual to a {\sc gofk} knot;
\end{itemize}
\item  $q= md+1$  such that $\gcd(m,d)=2$ and
\begin{itemize}
\item $k=\pm m $ so that $K$ is the dual to  a {\sc bgii} knot;
\end{itemize}
\item  $q=d(m-1)$  such that $d$ is odd and divides $m-1$ and either 
\begin{itemize}
\item $k=\pm dm $ so that $K$ is the dual to a {\sc bgiii} knot,
\item $k=\pm m $ so that $K$ is   the dual to a {\sc bgv} knot, or
\item $k =\pm 2m  = \pm 4m $ so that $K$ is  the dual to a {\sc spor} knot if $m=1-2d$; or  
\end{itemize}
\item $q= d(m-1)$ such that $d$ divides $2m+1$ and
\begin{itemize}
\item $k=\pm m $ or $k=\pm dm $ so that $K$ (in each case) is  the dual to a {\sc bgiv} knot.
\end{itemize}
\end{enumerate}
\end{thm}

Together Conjectures~\ref{conj:main} and \ref{conj:dp} assert that Theorem~\ref{thm:mainhomology} gives a complete list of knots in lens spaces with a longitudinal $S^{1} \times S^{2}$ surgery.   The possible lens spaces containing such knots are given in Theorem~\ref{thm:main}, but the homology classes of these knots have not yet been determined.  

\begin{conj}\label{conj:homologyclasses}
If a knot in a lens space $L$ represents the homology class $\kappa \in H_1(L)$ and admits a longitudinal surgery to $S^1 \times S^2$ then, up to homeomorphism, $L=L(m^2,q)$ and $\kappa=k\mu$ are as in some case of Theorem~\ref{thm:mainhomology}.
\end{conj}

\begin{remark}
We are informed that Cebanu has established this conjecture for $L$ as in the first two cases of Theorem~\ref{thm:mainhomology}, i.e.\ when $L \cong L(m^2,md+1)$ and $\gcd(m,d)=1,2$ \cite{cebanu}.
\end{remark}

\begin{remark}\label{rem:embeddings}
While Greene's work on lattice embeddings produced a classification of the homology classes of knots in lens spaces with a longitudinal $S^3$ surgery \cite{greene},  Lisca's work on lattice embeddings does not appear to produce information about the classification of homology classes of knots in lens spaces with longitudinal $S^1 \times S^2$ surgeries \cite{lisca}.  Nevertheless we examine a manner in which Lisca's lattice embeddings suggest knots in lens spaces with such surgeries in section~\ref{sec:embeddings}.
\end{remark}

\begin{remark}
In \cite{koyabel}, the authors classify the strongly invertible knots in $L((2n-1)^2, 2n)$ with longitudinal $S^1 \times S^2$ surgeries.  These belong to the first case of Theorem~\ref{thm:mainhomology} with $m=2n-1$ and $d=1$.  The knots are
 {\sc bgi} knots and may also be realized as {\sc gofk} knots.
\end{remark}

\begin{thm}\label{thm:reduction}
Conjecture~\ref{conj:homologyclasses} implies Conjecture~\ref{conj:dp}.

That is, confirming the list of homology classes of knots in lens spaces admitting a longitudinal $S^1 \times S^2$ surgery will confirm that the families {\sc bg}, {\sc gofk}, {\sc spor} together constitute all the doubly primitive knots in $S^1 \times S^2$.
\end{thm}

\begin{proof}
Because there is a unique simple knot for each homology class in a lens space, this theorem follows from Theorem~\ref{thm:simplewave}.
\end{proof}

\begin{remark}
Indeed, Conjecture~\ref{conj:main} may be rephrased as saying if a knot $K$ in a lens space $L$ admits a longitudinal $S^1 \times S^2$ surgery, then up to homeomorphism $L=L(p,q)$ and $[K]=\kappa = k\mu$ in some case of  Theorem~\ref{thm:mainhomology} and moreover $K = K(p,q,k)$.
\end{remark}

\subsection{Fibered knots and spherical braids}

Ni shows that knots in $S^3$ with a lens space surgery have fibered exterior \cite{ni}.  One expects the same to be true for any knot in $S^1 \times S^2$ with a lens space surgery.  Using knot Floer Homology, Cebanu shows this is indeed the case.  

\begin{thm}[Cebanu \cite{cebanu}]\label{thm:sphericalbraid2}
A knot in $S^1 \times S^2$ with a longitudinal lens space surgery has fibered exterior.
\end{thm}

Prior to learning of Cebanu's results, we had confirmed this for all our knots by showing they are spherical braids.  This is done in section~\ref{sec:braids}.  
A link in $S^1 \times S^2$ is a {\em (closed) spherical braid} if it is transverse to $\{\theta\} \times S^2$ for each $\theta \in S^1$.

\begin{thm}\label{thm:sphericalbraid}
In $S^1 \times S^2$, the knots in families {\sc bg}, {\sc gofk}, and {\sc spor} are all isotopic to spherical braids. 
\end{thm}

Braiding characterizes fiberedness for non-null homologous knots in $S^1\times S^2$.

\begin{lemma}\label{lem:sphericalbraid} 
A non-null homologous knot in $S^1 \times S^2$ has fibered exterior if and only if it is isotopic to a spherical braid.
\end{lemma}
\begin{proof}
If $K$ is a spherical braid, then the punctured spheres $(\{\theta\} \times S^2) - N(K)$ for $\theta \in S^1$ fiber the exterior of $K$.   Let $X(K) = S^1 \times S^2 - N(K)$ denote the exterior of a knot $K$ in $S^1 \times S^2$.  If $K$ is a non-null homologous in $S^1 \times S^2$, then the kernel of the map  $H_1(\bdry X(K) ) \to H_1(X(K), \bdry X(K))$ induced by inclusion is generated by a multiple of the meridian of $K$.  Hence if $X(K)$ is fibered, then the boundary of a fiber is a collection of coherently oriented meridional curves.  Therefore the trivial (meridional) filling of $X(K)$ which returns $S^1 \times S^2$ must also produce a fibration over $S^1$ by closed surfaces, the capped off fibers of $X(K)$.  Hence the fibers of $X(K)$ must be punctured spheres, and so $K$ is isotopic to a spherical braid.
\end{proof}

\begin{remark}
Together Theorem~\ref{thm:sphericalbraid2} and Lemma~\ref{lem:sphericalbraid} suggest the study of surgery on spherical braids.

One may care to compare these results with Gabai's resolution of Property R \cite{gabai:fato3mIII}: The only knot in $S^1 \times S^2$ with a surgery yielding $S^3$ is a fiber $S^1 \times \ast$.
\end{remark}

\subsection{Geometries of knots and lens space surgeries}
For completeness, here we address the classification of lens space Dehn surgeries on non-hyperbolic knots in $S^1 \times S^2$.

We say a knot $K$ in a $3$--manifold  $M$ is either spherical, toroidal, Seifert fibered, or hyperbolic if its exterior $M-N(K)$ contains an essential embedded sphere, contains an essential embedded torus, admits a Seifert fibration, or is hyperbolic respectively.
By Geometrization for Haken manifolds, a knot $K$ in $S^1 \times S^2$ is (at least) one of these.  By the Cyclic Surgery Theorem \cite{cgls}, if a non-trivial knot $K$ admits a non-trivial lens space surgery, then either the surgery is longitudinal or $K$ is Seifert fibered.

If $K$ is spherical, then one may find a separating essential sphere in the exterior of $K$.  Since $S^1 \times S^2$ is irreducible, $K$ is contained in a ball.  Therefore $K$ only admits a lens space surgery (in fact an $S^1 \times S^2$ surgery) if $K$ is the trivial knot, \cite{gabai:fato3mIII, gl}.

If $K$ is Seifert fibered then $K$ is a torus knot.  This follows from the classification of generalized Seifert fibrations of $S^1 \times S^2$, \cite{JankinsNeumann}.  Note that the exceptional generalized fiber of $M(-1;(0,1))$ is a regular fiber of $M(0;(2,1),(2,-1))$, and its exterior is homeomorphic to both the twisted circle bundle over the \mobius band and the twisted interval bundle over the Klein bottle.  For relatively prime integers $p,q$ with $p\geq 0$, we define a $(p,q)$--torus knot $T_{p,q}$ in $S^1 \times S^2$ to be a regular fiber of the generalized Seifert fibration $M(0;(p,q),(p,-q))$. (The exceptional fibers may be identified with $S^1 \times \mbox{ \sc n}$ and $S^1 \times \mbox{ \sc s}$ for antipodal points $ \mbox{ \sc n},  \mbox{ \sc s} \in S^2$.)
Equivalently, we may regard $T_{p,q}$ as a curve on $\bdry N(S^1 \times \ast)$, $\ast \in S^2$, that is homologous to $q \mu + p \lambda \in H_1(\bdry N(S^1 \times \ast))$ for meridian-longitude classes $\mu, \lambda$ and an appropriate choice of orientation on $T_{p,q}$.  Observe $[T_{p,q}] = p[S^1 \times \ast] \in H_1(S^1 \times S^2)$ and $T_{p,q} = T_{p, q+Np}$ for any integer $N$.
Following \cite{moser, gordonsatellite} (though note that on the boundary of a solid torus, a $(p,q)$ curve for them is a $(q,p)$ curve for us)  any non-trivial lens space surgery on a $(p,q)$--torus knot with $p \geq 2$ has surgery slope $1/n$, taken with respect to the framing induced by the Heegaard torus,  and yields the lens space $L(np^2,npq+1)$.

If $K$ is toroidal and admits a non-trivial lens space surgery then the proof in \cite{bleilerlitherland} applies basically unaltered (because Seifert fibered knots in $S^1 \times S^2$ are torus knots and lens spaces are atoroidal) to show $K$ must be a $(2,\pm1)$--cable of a torus knot, where the cable is taken with respect to the framing on the torus knot induced by the Heegaard torus.   If $K$ is the $(2,\pm1)$--cable of the $(p,q)$--torus knot, then $\pm1$ surgery on $K$ with respect to its framing as a cable is equivalent to $\pm1/4$ surgery on the $(p,q)$--torus knot and thus yields $L(4p^2,4pq\pm1)$ or its mirror.

Hence we have:
\begin{thm}
A non-hyperbolic knot in $S^1 \times S^2$ with a non-trivial lens space surgery is either a $(p,q)$--torus knot or a $(2,\pm1)$--cable of a $(p,q)$--torus knot.
\end{thm}

Because the $(2,1)$--torus knot in $S^1 \times S^2$ contains an essential Klein bottle in its exterior, it is toroidal.
\begin{cor}\
\begin{itemize}
\item The smallest order lens space obtained by surgery on a toroidal knot in $S^1 \times S^2$ is homeomorphic to $L(4,1)$.  
The surgery dual is the simple knot $K(4,1,2)$.
\item The smallest order lens space obtained by surgery on a non-torus, toroidal knot in $S^1 \times S^2$ is homeomorphic to $L(16,9)$.
The surgery dual is the (unoriented) simple knot $K(16,9,4)$.
\end{itemize}
\end{cor}
(The orientation of a knot has no bearing upon its surgeries.  Ignoring orientations, $K(16,9,4)$ is equivalent to $K(16,9,12)$.  The simple knot $K(4,1,2)$ is isotopic to its own orientation reverse.)

\smallskip

Bleiler-Litherland conjecture that the smallest order lens space obtained by surgery on a hyperbolic knot in $S^3$ is homeomorphic to $L(18,5)$ \cite{bleilerlitherland}.  Among our list of doubly primitive knots in $S^1 \times S^2$, up to homeomorphism, we find three hyperbolic knots of order $5$ in families {\sc bgiii}, {\sc bgv}, and {\sc gofk}  with integral lens space surgeries; all the doubly primitive knots with smaller order are non-hyperbolic. 

\begin{conj}
Up to homeomorphism, $L(25,7)$ and $L(25,9)$ are the smallest order lens spaces obtained by surgery on a hyperbolic knot in $S^1 \times S^2$.  Moreover, the surgeries occur on the knots shown in Figure~\ref{fig:smallhyperbolicknots}.  From left to right, their surgery duals are the (unoriented) simple knots  $K(25,7,5)$, $K(25,7,10)$, and $K(25,9,10)$ respectively.
\end{conj}

\begin{remark}
The knots on the right and left of Figure~\ref{fig:smallhyperbolicknots} are actually isotopic.  Kadokami-Yamada show that among the non-torus {\sc gofk} knots this is the only one (up to homeomorphism) that admits two non-trivial lens space surgeries \cite{ky}.  Along these lines, Berge shows there is a unique hyperbolic knot in the solid torus with two non-trivial lens space surgeries \cite{bergeS1xD2}, and this gives rise to a single {\sc bgiv} knot (up to homeomoprhism) having surgeries to both orientations of $L(49,18)$ \cite{bhw}.
\end{remark}

\begin{figure}
\centering
\includegraphics[width=5in]{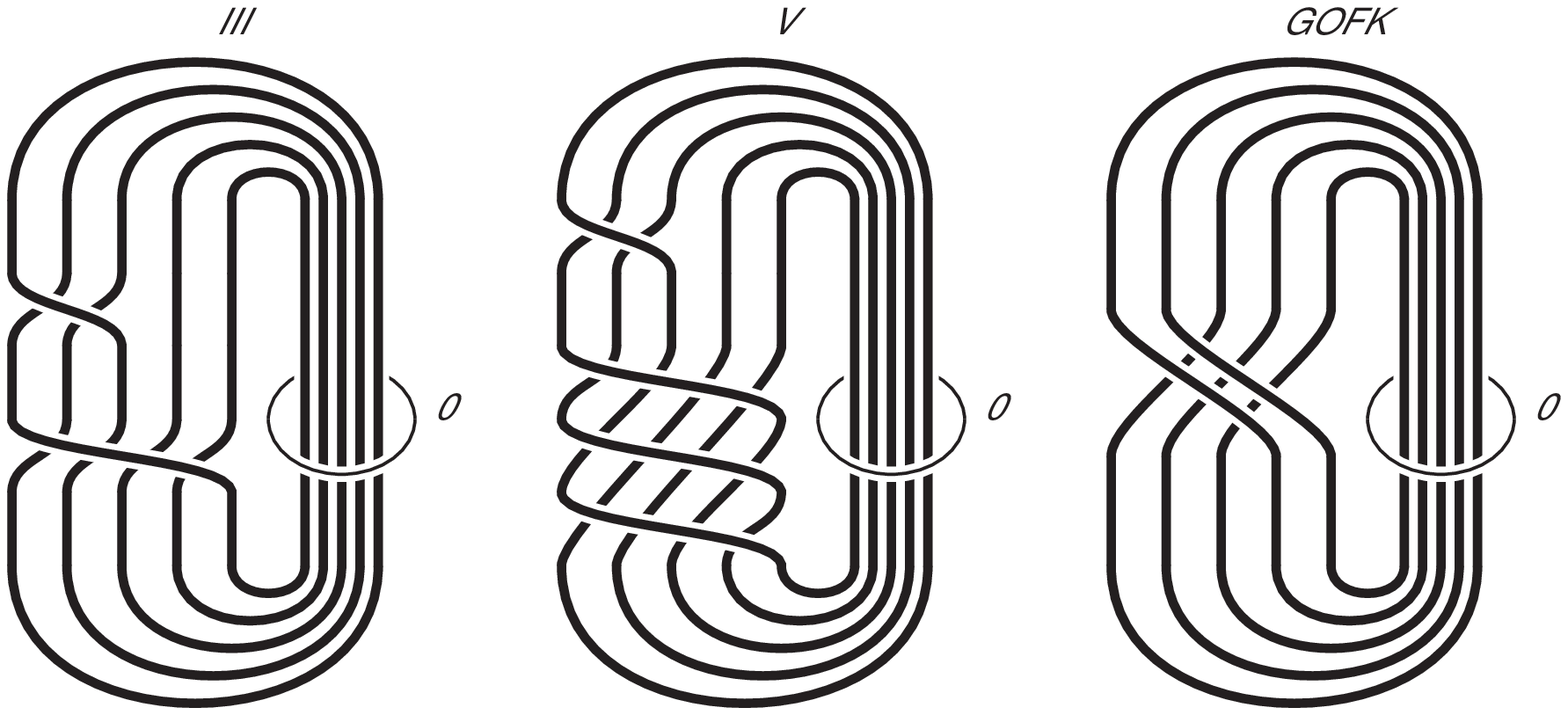}
\caption{ }
\label{fig:smallhyperbolicknots}
\end{figure}

\subsection{Basic definitions and some notation}\label{sec:notation}

\subsubsection{Dehn surgery}
Consider a knot $K$ in a closed $3$--manifold $M$ with regular solid torus neighborhood $N(K)$.  The isotopy classes of essential simple closed curves on $\bdry N(K) \cong \bdry(M-N(K))$ are called {\em slopes}.  The {\em meridian} of $K$ is the slope that bounds a disk in $N(K)$ while the slopes that algebraically intersect the meridian once (and hence are isotopic to $K$ in $N(K)$) are {\em longitudes}.  Given a slope $\gamma$, the manifold obtained by removing the solid torus $N(K)$ from $M$ and attaching another solid torus so that $\gamma$ is its meridian is the result of {\em $\gamma$--Dehn surgery on $K$}.  The core of the attached solid torus is a new knot in the resulting manifold and is the {\em surgery dual} to $K$.  If $\gamma$ is a longitude, then $\gamma$--Dehn surgery is a {\em longitudinal surgery} or simply a {\em surgery}.  Fixing a choice of longitude and orienting both the meridian and this longitude so that they represent homology classes $\mu$ and $\lambda$ in $H_1(\bdry N(K))$ with $\mu \cdot \lambda = +1$ enables a parametrization associating the slope $\gamma$ to the extended rational number $\frac{p}{q} \in \Q \cup \{1/0\}$, $\gcd(p,q)=1$, if for some orientation $[\gamma] = p\lambda + q\mu$.  Then $\gamma$--Dehn surgery may also be denoted $\frac{p}{q}$--Dehn surgery.  Consequentially longitudinal surgery is also called {\em integral surgery}.

\subsubsection{Tangles and bandings}
The knot $K$ in $M$ is said to be {\em strongly invertible} if there is an involution $u$ on $M$ that set-wise fixes $K$ and whose fixed set intersects $K$ exactly twice, and the involution is said to be a {\em strong involution}.  The quotient of $M$ by $u$ is a $3$--manifold $M/\sim$, where $x \sim u(x)$, in which the fixed set of $u$ descends to a link $J$ and the knot $K$ descends to an embedded arc $\alpha$ such that $J \cap \alpha = \bdry \alpha$.  A small ball neighborhood $B = N(\alpha)$ of $\alpha$ intersects $J$ in a pair of arcs $t$ so that $(B, t)$ is a {\em rational tangle}, i.e. a tangle in a ball homeomorphic to $(D^2 \times I, \{\pm\frac{1}{2}\} \times I)$ where $D^2$ is the unit disk in the complex plane and $I$ is the interval $[-1,1]$. The solid torus neighborhood $N(K)$ of $K$ may be chosen so that the image of its quotient under $u$ is $B$, and equivalently so that it is the double cover of $B$ branched along $t$.  The {\em Montesinos Trick} refers to the correspondence through branched double covers and quotients by strong involutions between replacing the rational tangle $(B,t)$ with another and Dehn surgery on $K$.  In particular, a {\em banding} of $J$ along the arc $\alpha$ corresponds to longitudinal surgery on $K$.  A banding of $J$ is the act of embedding of a rectangle $I \times I$ in $M/\sim$ to meet $J$ in the pair of opposite edges $I \times \bdry I$ and exchanging those sub-arcs of $J$ for the other pair of opposite edges $\bdry I \times I$.  The banding occurs along an arc $\alpha = \{0\} \times I$ and the banding produces the dual arc $I \times \{0\}$.  Figure~\ref{fig:bandingillustration} illustrates the banding operation and both its framed arc and literal band depictions that we use in this article.  Figure~\ref{fig:twistillustration} shows how a rectangular box labeled with an integer $n$ denotes a sequence of $|n|$ twists in the longer direction. The twists are right handed if $n>0$ and left handed if $n<0$.  To highlight cancellations, a pair of twist boxes will be colored the same if their labels have opposite sign.

\begin{figure}
\begin{minipage}[b]{0.5\textwidth}
\centering
\includegraphics[height=2in]{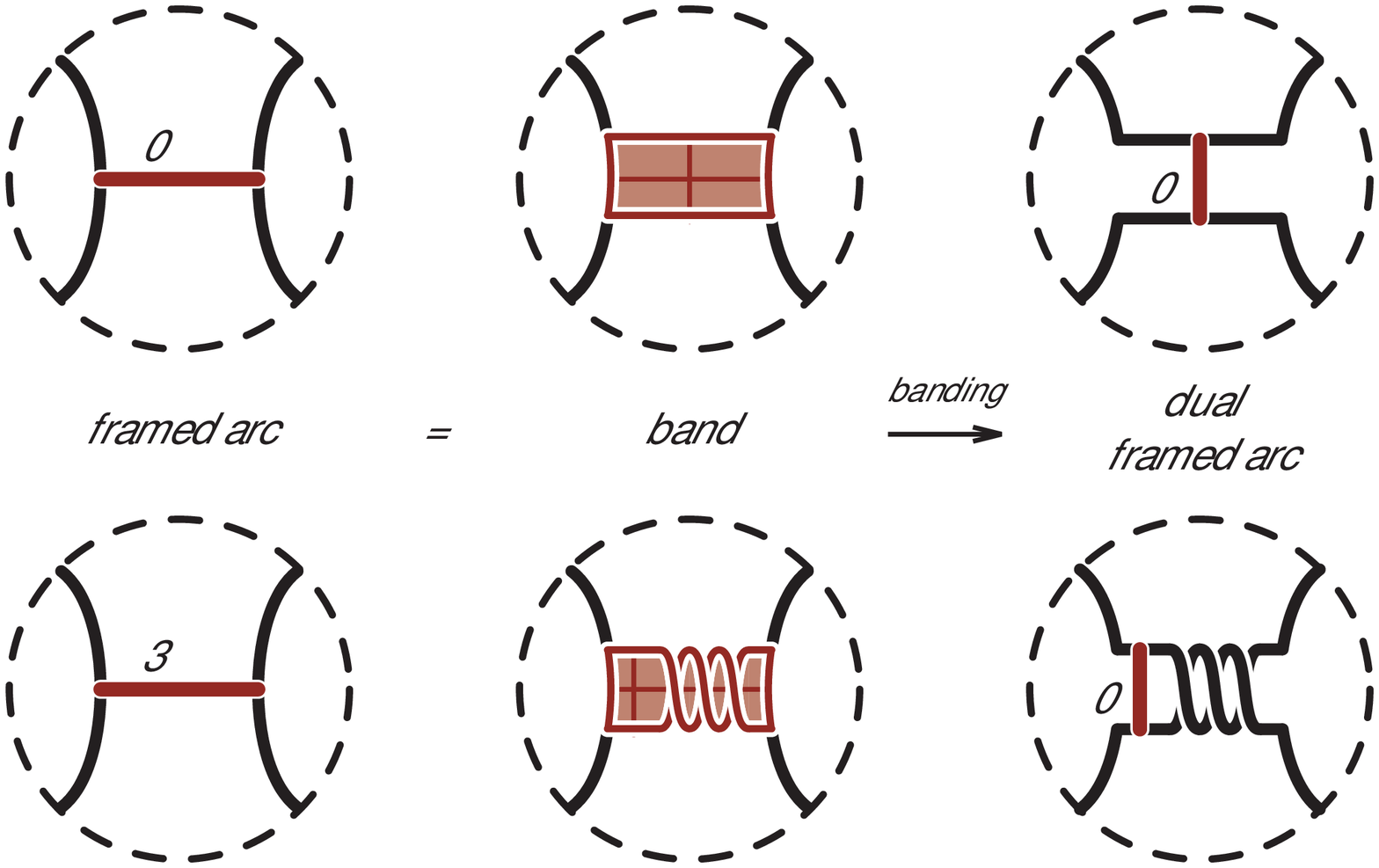}
\caption{}
\label{fig:bandingillustration}
\end{minipage}
\hspace{0.5cm}
\begin{minipage}[b]{0.5\textwidth}
\centering
\includegraphics[height=2in]{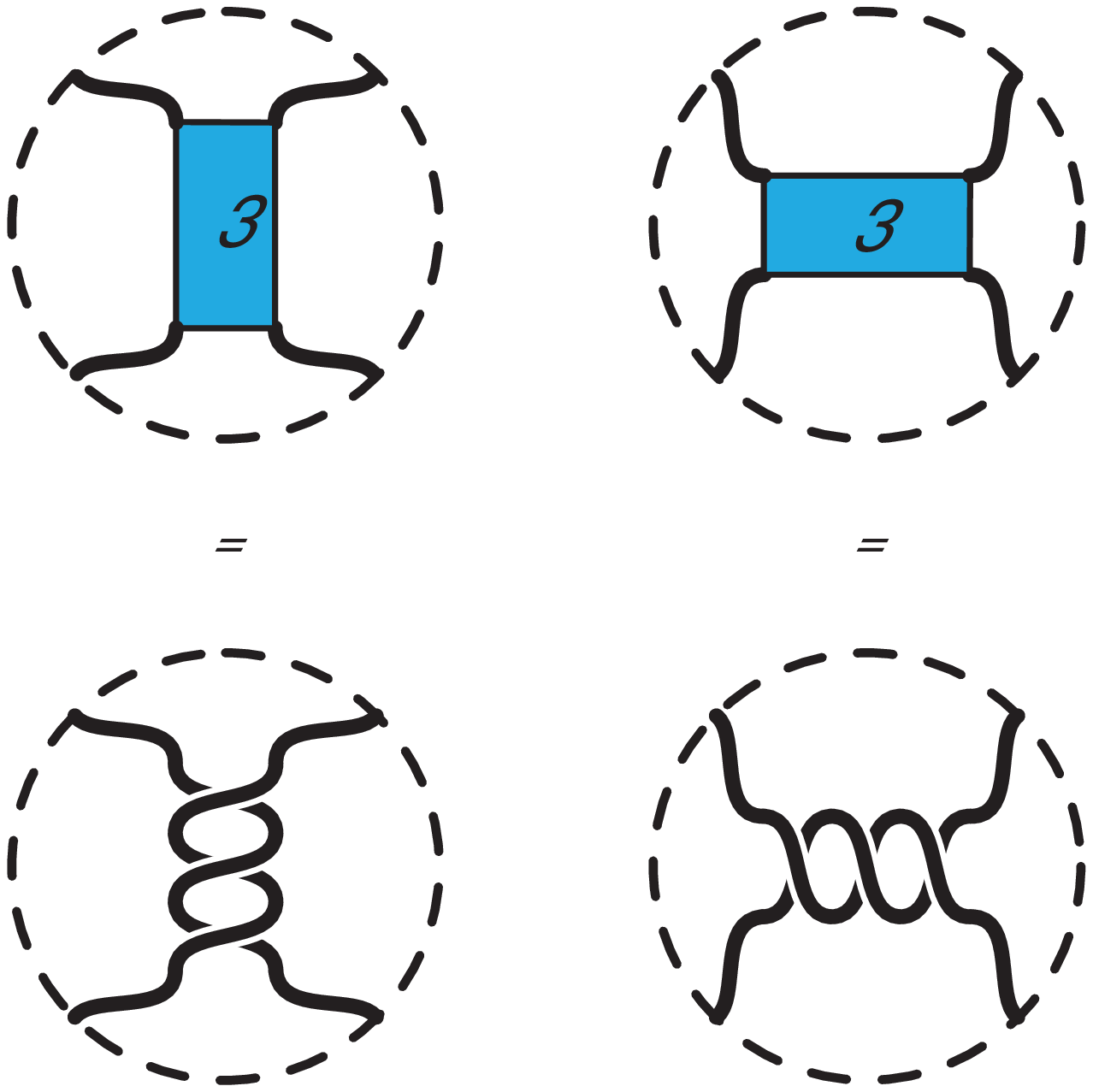}
\caption{}
\label{fig:twistillustration}
\end{minipage}
\end{figure}

\subsubsection{Lens spaces, two bridge links, plumbing manifolds}
The lens space $L(p,q)$ is defined as the result of $-p/q$--Dehn surgery on the unknot in $S^3$.
The lens space $L(p,q)$ may be obtained by surgery on the linear chain link as shown at the top of Figure~\ref{fig:chainto4plat} with integral surgery coefficients $-a_1, \dots, -a_n$ that are the negatives of the coefficients of a continued fraction expansion 
\[ \frac{p}{q} = [a_1, \dots, a_n]^- = a_1 - \cfrac{1}{a_2 - \cfrac{1}{ \quad \overset{\ddots}{a_{n-1}- \cfrac{1}{a_n}}}}.\]
The picture of this chain link also shows the axis of a strong involution $u$ that extends through the surgery to an involution of the lens space.  The quotient of this involution of the lens space, via the involution of this surgery diagram, is $S^3$ in which the axis descends to the two-bridge link $K(p,q)$ with the diagram $L(-a_1, \dots, -a_n)$ as shown at the bottom of Figure~\ref{fig:chainto4plat}.  The orientation preserving double cover of $S^3$ branched over $K(p,q)$ is the lens space $L(p,q)$.  
Observe that the two-component unlink is $K(0,1)$, $0/1 =[0]^-$, and we regard $S^1 \times S^2$ as the lens space $L(0,1)$.

\begin{figure}
\centering
\includegraphics[width=4in]{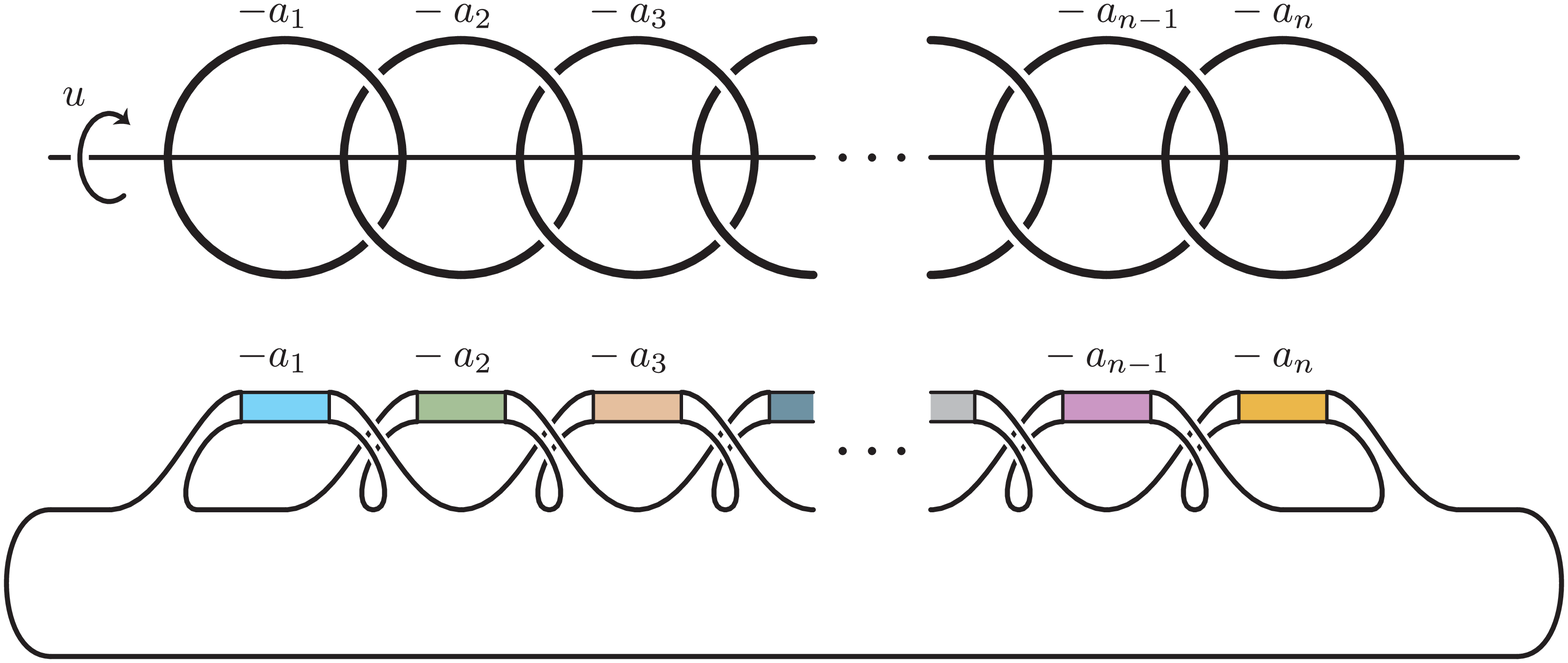}
\caption{}
\label{fig:chainto4plat}
\end{figure}

\begin{figure}
\centering
\includegraphics[width=5in]{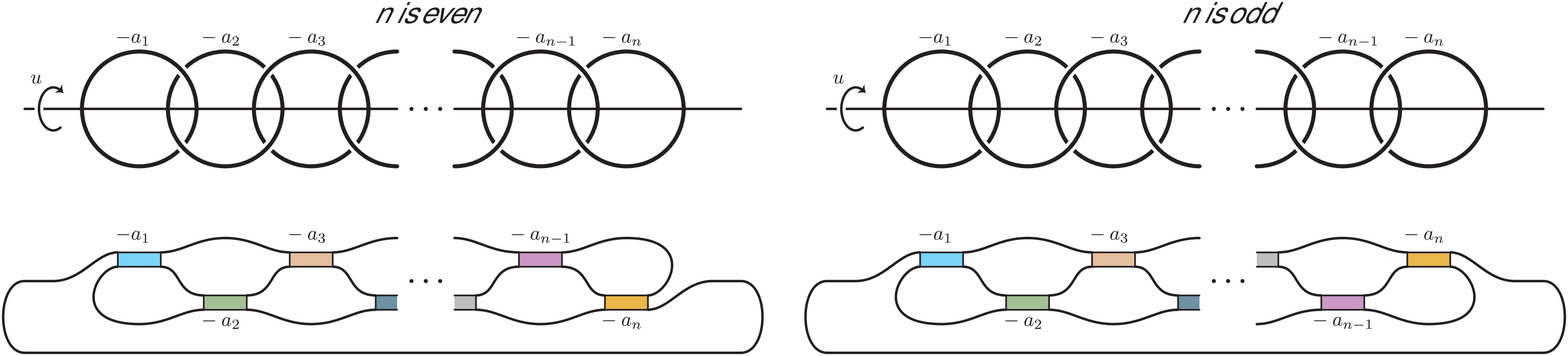}
\caption{}
\label{fig:chainto4platv2}
\end{figure}

From a $4$--manifold perspective, the top of Figure~\ref{fig:chainto4plat} is a Kirby diagram for a plumbing manifold whose boundary is $L(p,q)$.  By an orientation preserving homeomorphism, we may take $p>q>0$ and restrict the continued fraction coefficients to be integers $a_i \geq 2$ so that $L(p,q)$ is the oriented boundary of the negative definite plumbing manifold $P(p,q)$ associated to the tuple $(-a_1, \dots, -a_n)$.

Figure~\ref{fig:chainto4platv2} shows alternative (and isotopic) versions of this chain link and two-bridge link diagrams for the two cases of $n$ even and $n$ odd.

\subsection{Acknowledgements}
The authors would like to thank John Berge, Radu Cebanu, and Joshua Greene. 

This work is partially supported by grant \#209184 to Kenneth L.\ Baker from the Simons Foundation, by the Spanish GEOR MTM2011-22435 to Ana G.\ Lecuona, and EPSRC grants G039585/1 and H031367/1 to Dorothy Buck.

\section{Generalizing Berge's doubly primitive knots}
Berge describes twelve families of doubly primitive knots in $S^3$, \cite{bergeII}.  Greene confirms that this list is complete, \cite{greene}.  The first author gives surgery descriptions of these knots and tangle descriptions of the quotients by their strong involutions, \cite{bakerI, bakerII}.  (These knots admit unique strong involutions, \cite{wangzhou}.) 

 We partition Berge's twelve families into three broader families:  The Berge-Gabai knots, family {\sc bg}, arising from knots in solid tori with longitudinal surgeries producing solid tori.   The knots that embed in the fiber of a genus one fibered knot (the figure eight knot or a trefoil), family {\sc gofk}.  The so-called sporadic knots, family {\sc spor}, which may be seen to embed in a genus one Seifert surface of a banding of a $(2,\pm1)$--cable of a trefoil.  (The framing of this cabling is with respect to the Heegaard torus containing the trefoil.)

Here, we generalize these three families of doubly primitive knots in $S^3$ to obtain three analogous families of doubly primitive knots in $S^1 \times S^2$ that we also call {\sc bg}, {\sc gofk}, and {\sc spor}.   We provide explicit descriptions of these using tangle descriptions.

\subsection{The {\sc bg} knots}
Let us say a {\em strong involution} of a knot in a solid torus is an involution of a solid torus whose fixed set is two properly embedded arcs such that the knot intersects this fixed set twice and is invariant under the involution.  The quotient of the pair of the solid torus and fixed set under this involution is a rational tangle $(B,t)$ where $B$ is a $3$--ball and $t$ is a pair of properly embedded arcs together isotopic into $\bdry B$.  The image of the knot in this quotient is an arc $\alpha$ embedded in $B$ with $\bdry \alpha = \alpha \cap t$.  A knot with a strong involution is {\em strongly invertible}.

The Berge-Gabai knots in solid tori are all strongly invertible as evidenced by them being $1$--bridge braids (or torus knots) and hence tunnel number $1$.  We say an arc $\alpha$ in a rational tangle is a {\em Berge-Gabai arc} if it is the image of a Berge-Gabai knot under the quotient by a strong involution. 
By virtue of the Berge-Gabai knots admitting longitudinal solid torus surgeries, there are bandings along these arcs that produce rational tangles.   The dual arcs to these bandings in the new rational tangles are also Berge-Gabai arcs.  Using these we obtain a proof of Theorem~\ref{thm:main} along the lines of Rasmussen's observation.

\begin{proof}[Tangle Proof of Theorem~\ref{thm:main}]
Doubling a rational tangle (by gluing it to its mirror) produces the two-component unlink $K(0,1)$.  Therefore, if $(B,t)$ is a rational tangle and $\alpha$ is a Berge-Gabai arc in $(B,t)$ along which a banding produces the rational tangle $(B,t')$, then banding the double $(B,t) \cup -(B,t) = (S^3, K(0,1))$ along $\alpha$ produces the two-bridge link $(B,t') \cup -(B,t)$.

In \cite{bakerbuck}, descriptions of the Berge-Gabai arcs are derived from the quotient tangle descriptions in \cite{bakerII} of the knots in  Berge's doubly primitive families I -- VI \cite{bergeII}.      As done in \cite{bakerbuck}, one may then explicitly observe that family VI is contained within family V, and (with mirroring) family V is dual to family III.  Families I, II, and IV are each self-dual.  (One family of Berge-Gabai arcs is dual to another if the arc dual to the banding along any arc in the first family, together with the resulting tangle, may be isotoped while fixing the boundary of the tangle into the form of a member of the second family.)   It is also shown in \cite{bakerbuck} that these knots in solid tori admit a unique strong involution in which the solid torus quotients to a ball, and hence up to homeomorphism there is a unique arc in a rational tangle corresponding to each Berge-Gabai knot in the solid torus. The resulting classification of bandings between rational tangles up to homeomorphism from \cite{bakerbuck} is shown in Figure~\ref{fig:RationalTangleRSR}.   
Figures~\ref{fig:BGIandII}, \ref{fig:BGIIIandV}, and \ref{fig:BGIVandIV2}
 show the result of doubling these families (with their duals) to obtain $K(0,1)$ and then banding along a Berge-Gabai arc to obtain a two-bridge link.  Observe that the two-bridge links produced by the dual pairs in families III and V in Figures~\ref{fig:BGIandII} are equivalent as are the two-bridge links produced by the dual pairs in family IV.   For the purposes of proving Theorem~\ref{thm:main} only one among each of these dual pairs is required.
 
\begin{figure}
\centering
\includegraphics[width=5.5in]{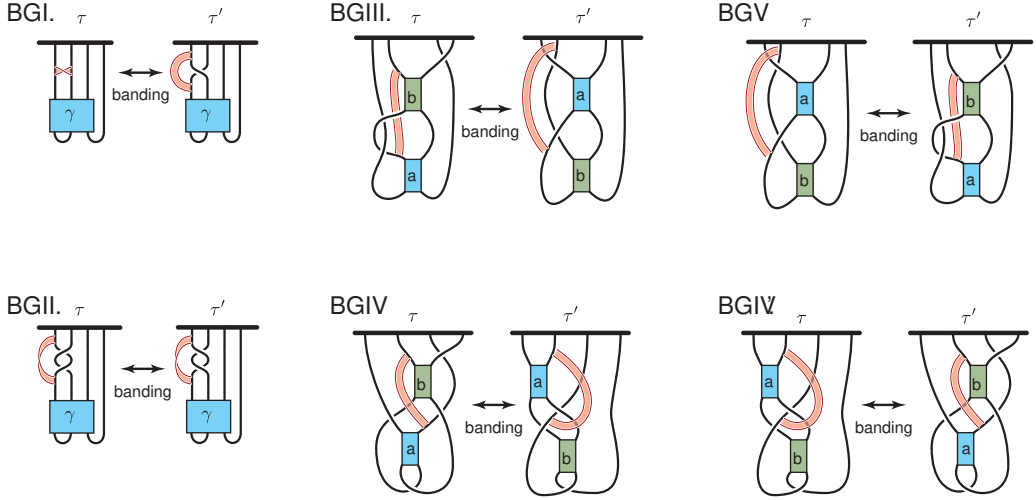}
\caption{The bandings between rational tangles up to homeomorphism, corresponding to the double branched covers of the Berge-Gabai knots in solid tori.  For {\sc bgi} and {\sc bgii} $\gamma$ is a $3$--braid.}
\label{fig:RationalTangleRSR}
\end{figure}

\begin{figure}
\centering
\includegraphics[width=5.5in]{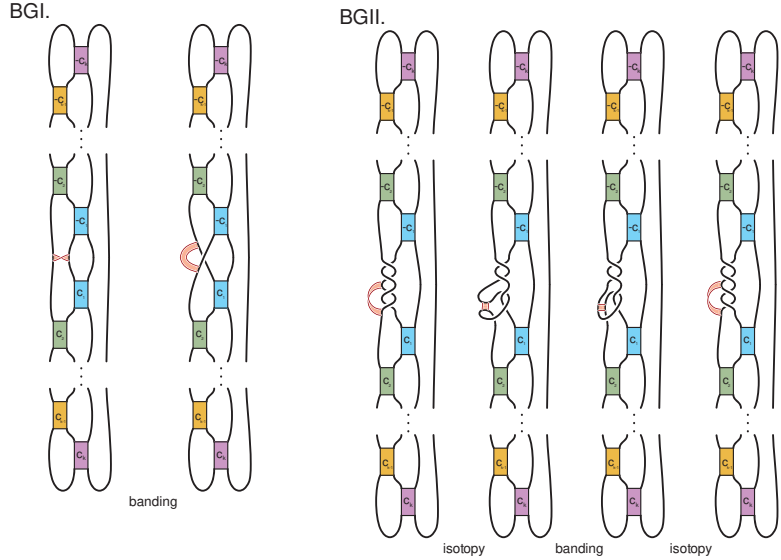}
\caption{Left, a banding from the unlink to $K(p,q)$ where $\frac{p}{q} = [c_k, \dots,  c_1, -1, -c_1,  \dots, -c_k]^-$ corresponding to family {\sc bgi}. Right, a banding from the unlink to $K(p,q)$ where $\frac{p}{q} = [c_k,\dots,  c_1, 4, -c_1,  \dots, -c_k]^-$ corresponding to family {\sc bgii}.}
\label{fig:BGIandII}
\end{figure}

\begin{figure}
\centering
\includegraphics[width=5.5in]{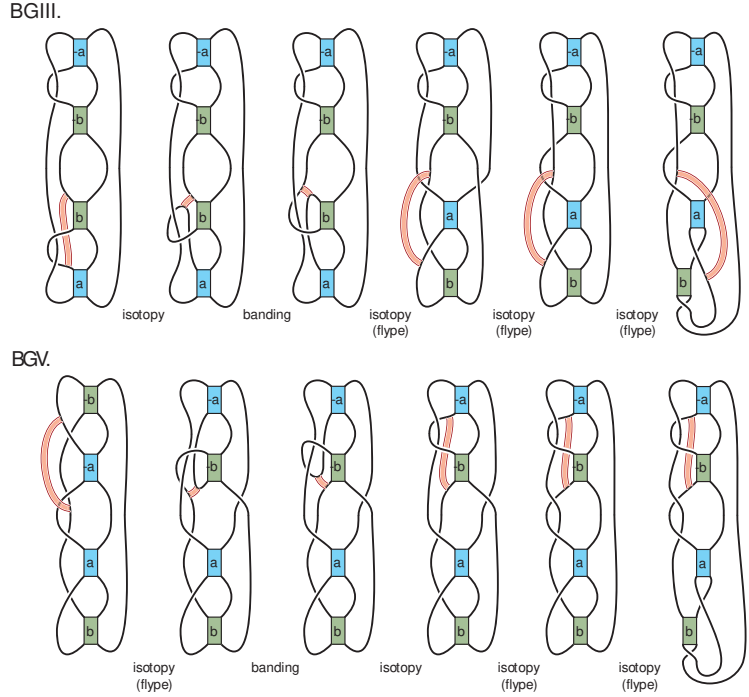}
\caption{Two bandings between the unlink and $K(p,q)$ where $\frac{p}{q} = [a,2,b,2,-a+1,-b+1]^-$ corresponding to families {\sc bgiii} and {\sc bgv}.}
\label{fig:BGIIIandV}
\end{figure}

\begin{figure}
\centering
\includegraphics[width=5.5in]{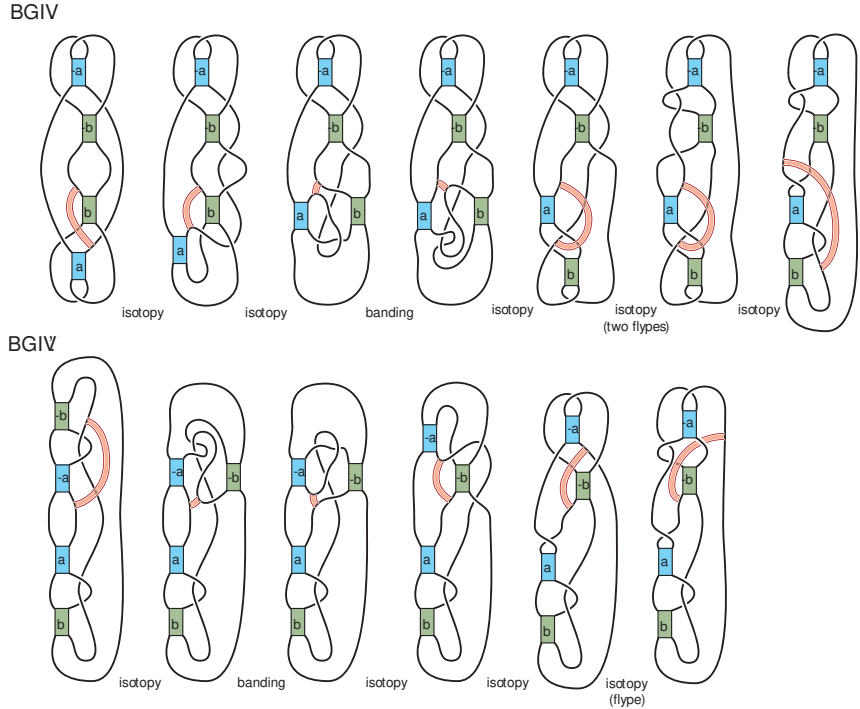}
\caption{Two bandings between the unlink and $K(p,q)$ where $\frac{p}{q}=[a-1,-2,b,-a+1,2,-b]^-$ both corresponding to family {\sc bgiv}.}
\label{fig:BGIVandIV2}
\end{figure}

\begin{figure}
\centering
\includegraphics{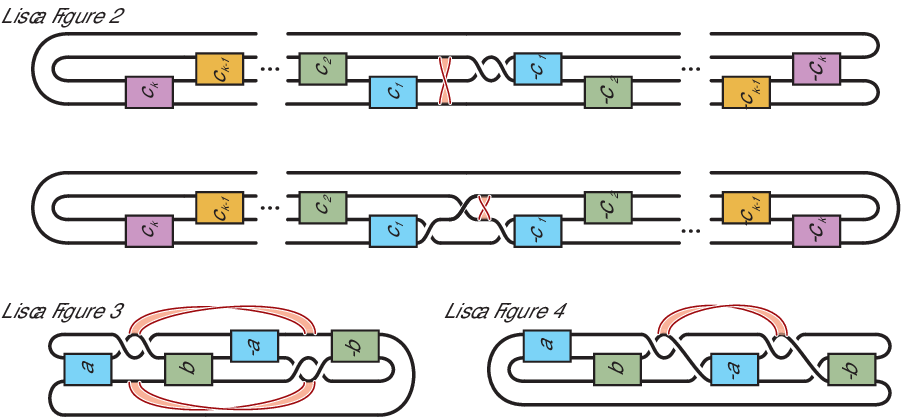}
\caption{}
\label{fig:liscalinks}
\end{figure}

We now observe that the two-bridge links produced match with those of Lisca's Section~8 \cite{lisca}.  He shows that up to homeomorphism these links may be presented as in the first of one of his Figure~2, Figure~3, or Figure~4.    We redraw these three in Figure~\ref{fig:liscalinks} for the reader's convenience, isotoping his Figure~2.   In each of his Figure~2 and Figure~4 Lisca exhibits a single banding as shown in our Figure~\ref{fig:liscalinks} that transforms those two-bridge links to the unlink.  In his Figure~3 he uses two bandings as also shown in our Figure~\ref{fig:liscalinks}. As one may now observe, the two bridge links of Lisca's Figures~2,3,4 correspond (with mirroring and reparametrizations as needed) to those produced respectively in families I, IV, III  of Figures~\ref{fig:BGIandII}, \ref{fig:BGIVandIV2}, and \ref{fig:BGIIIandV}.  Note that family II produces two-bridge links not accounted for in Lisca's pictures.  Nevertheless the corresponding lens spaces are accounted for in his proof of \cite[Lemma~7.2]{lisca}, as discussed in Remark~\ref{rem:missing}.
\end{proof}

\begin{remark}
In $S^1 \times S^2$, as in $S^3$, family I consists of the torus knots while family II consists of the $(2,\pm1)$--cables of torus knots.  (This cable is taken with respect to the framing induced by the Heegaard torus containing the torus knot.)  Families III, IV, and V contain hyperbolic knots.
\end{remark}

\subsection{The {\sc gofk} knots}
Conjecture~1.8 of \cite{greene} proposes that the knot surgeries corresponding to the double branched covers of the above bandings are, up to homeomorphisms, the only way that integral surgery on a knot in $S^1 \times S^2$ may yield a lens space.   However since $S^1 \times S^2$ contains a genus one fibered knot, we may form the family {\sc gofk} of knots that embed in the fiber of genus one fibered knots in $S^1 \times S^2$ and then mimic \cite{bakerI} to produce our first infinite family of counterexamples.

The annulus together with the identity monodromy gives an open book for $S^1 \times S^2$.  Plumbing on a positive Hopf band along a spanning arc produces a once-punctured torus open book, i.e.\ a (null-homologous) genus one fibered knot.  One may show (e.g.\ \cite{cgofkils}) that this and its mirror are the only two genus one fibered knots in $S^1 \times S^2$.  Any essential simple closed curve in one of these fibers is then a doubly primitive knot in $S^1 \times S^2$ and thus admits a lens space surgery along the slope of its page framing.  We call the family of these essential simple closed curves the {\sc gofk}.  These knots are analogous to the knots in Berge's families VII and VIII, \cite{bergeII}.  Yamada had previously developed this family of knots with these lens space surgeries \cite{yamada}, though  he constructs them from a different viewpoint. 

\begin{lemma}\label{lem:gofk}
There are {\sc gofk} knots that are not Berge-Gabai knots.  Moreover the {\sc gofk} knots contain hyperbolic knots of arbitrarily large volume.
\end{lemma}
\begin{proof}
Following \cite{bakerI} each {\sc gofk} knot admits a surgery description on the Minimally Twisted $2n+1$ Chain link  (the MT$(2n+1)$C for short) for some $n \in \Z$.  Furthermore, for each positive integer $n$ and any value $N$, there is a doubly primitive knot on a once-punctured torus page of this open book with a surgery description on the MT$(2n+1)$C whose surgery coefficients all have magnitude greater than $N$.  Therefore, as in \cite{bakerI}, since MT$(2n+1)$C is hyperbolic for $n\geq 2$ we may conclude using Thurston's Hyperbolic Dehn Surgery Theorem \cite{thurston} and the lower bound on a hyperbolic manifold with $n$ cusps \cite{adams} that the set of volume of hyperbolic knots on this once-punctured torus page is unbounded.  The Berge-Gabai knots in $S^1 \times S^2$ of \cite[Conjecture~1.8]{greene} however all admit surgery descriptions on the MT5C (as apparent from \cite{bakerII})  and thus have volume less than $\vol(MT5C)<11$.
\end{proof}

Any genus one fibered knot may be viewed as the lift of the braid axis in the double cover of $S^3$ branched over a closed $3$--braid.  This enables a pleasant interpretation of the {\sc gofk} knots and their lens space surgeries as corresponding to bandings from a closed $3$--braid presentation of the unlink to two-bridge links.  Because Lisca's list of two-bridge links that admit bandings to the unlink is complete, these bandings must give different bandings to the unlink for some (in a sense, most) two-bridge links.  Indeed Figure~\ref{fig:twobandings} shows the two different bandings between the unlink and a two-bridge link corresponding to the {\sc gofk} knots on the left and {\sc bgi} knots on the right.

\begin{figure}
\centering
\includegraphics[width=5in]{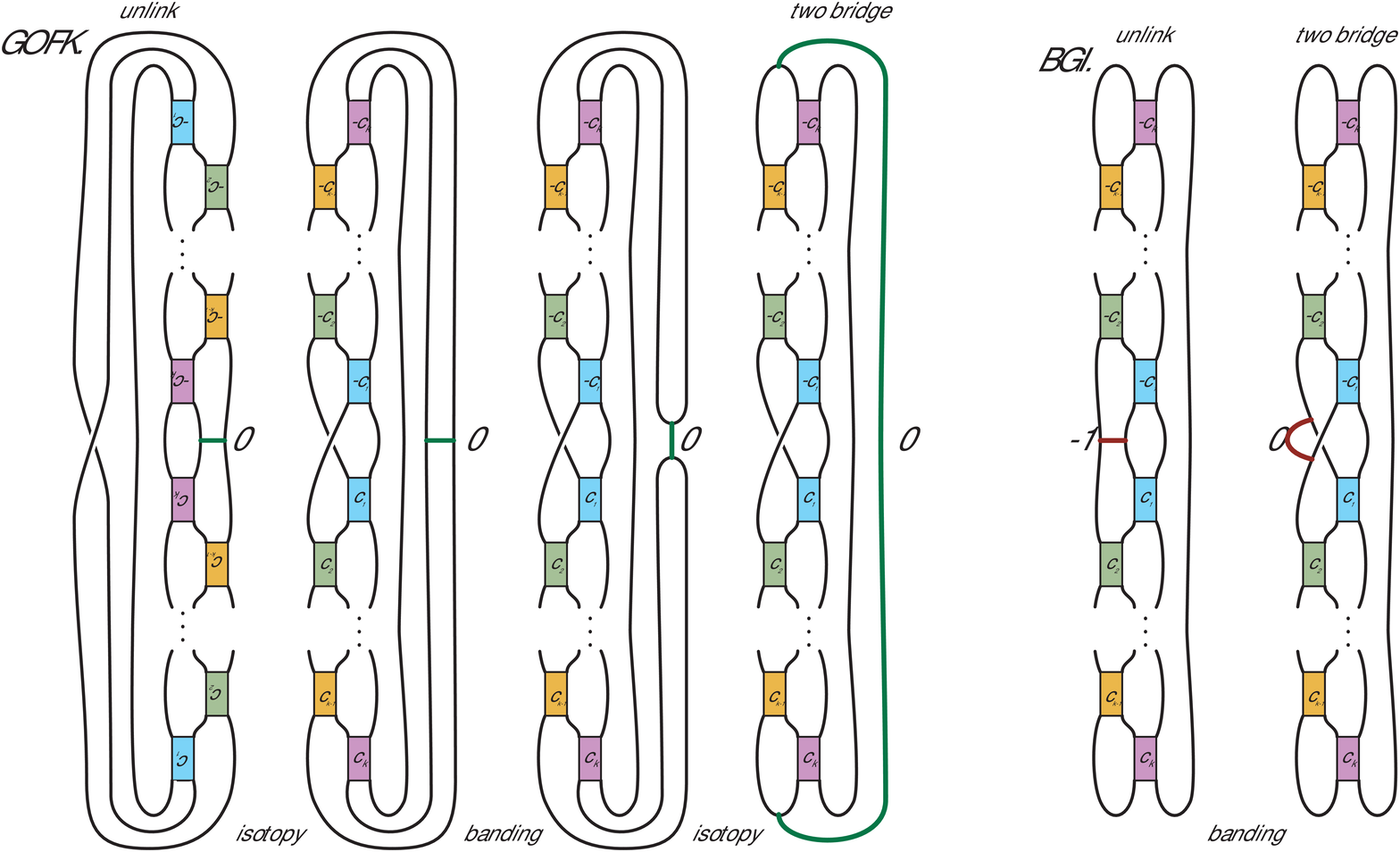}
\caption{Left, a banding from the unlink to a two-bridge link corresponding to family {\sc gofk}.  Right, a banding from the unlink to the same two-bridge link corresponding to family {\sc bgi}.}
\label{fig:twobandings}
\end{figure}

\subsection{The {\sc spor} knots}
Berge's families IX--XII of doubly primitive knots in $S^3$ condense to two families and are collectively referred to as the sporadic knots.  In the double branched cover, the family of blue arcs ($n \in \Z$) in the second link of Figure~\ref{fig:sporadic} lifts to the analogous sporadic knots in $S^1 \times S^2$, the $S^1\times S^2$--{\sc spor} knots.   (As one may confirm by examining the tangle descriptions in \cite{bakerII}, Berge's two sporadic knot families are obtained by placing $-3$ instead of $-2$ twists in either the top or bottom dashed oval, but not both.)  This second link is the two-component unlink as illustrated by the subsequent isotopies.  The link at the beginning of Figure~\ref{fig:sporadic} results from banding as shown.  It is a two-bridge link and coincides with the two bridge link in family III of Figure~\ref{fig:BGIIIandV} with $a=n$ and $b=-1$ and, after mirroring, with the two-bridge link of Lisca's Figure 4 in our Figure~\ref{fig:liscalinks} with $a=2$ and $b=-n+1$.   In Section~\ref{sec:homologyclasses} we show these knots are generically distinct from the Berge-Gabai knots by examining the homology classes of the corresponding knots in the lens spaces.

\begin{figure}
\centering
\includegraphics[width=5in]{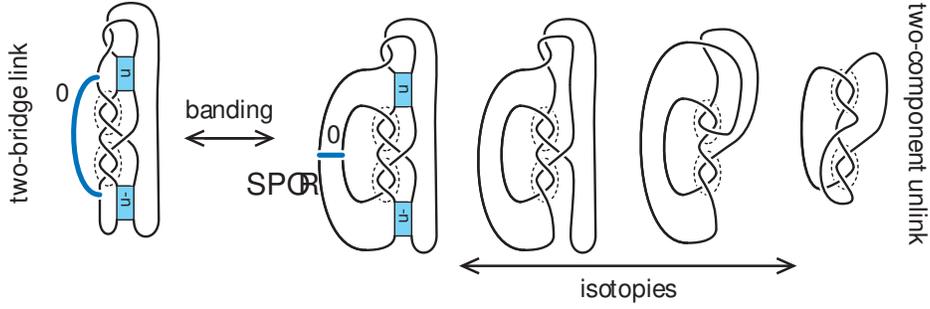}
\caption{A banding between a two-bridge link and the unlink corresponding to the sporadic family of knots.  Replace either of the $-2$ twists in the dashed ovals, but not both, with $-3$ twists to obtain bandings corresponding to Berge's families of sporadic knots in $S^3$. }
\label{fig:sporadic}
\end{figure}

\section{Homology classes of the dual knots in lens spaces} 
\label{sec:homologyclasses} 
 
Figure~\ref{fig:chainsurgery} gives, up to homeomorphism, strongly invertible surgery descriptions of the lens space duals to the {\sc bg}, {\sc gofk}, and {\sc spor} knots with their $S^1 \times S^2$ surgery coefficient.   Appendix section~\ref{sec:chaintotangle} shows how the quotients of these surgery descriptions produce the tangles  in Figures~\ref{fig:BGIandII}, \ref{fig:BGIIIandV}, \ref{fig:BGIVandIV2}, \ref{fig:twobandings}, and \ref{fig:sporadic} that defined these knots.   We will use these surgery descriptions to determine the homology classes of these knots.

 \begin{figure}
\centering
\includegraphics[width=5.5in]{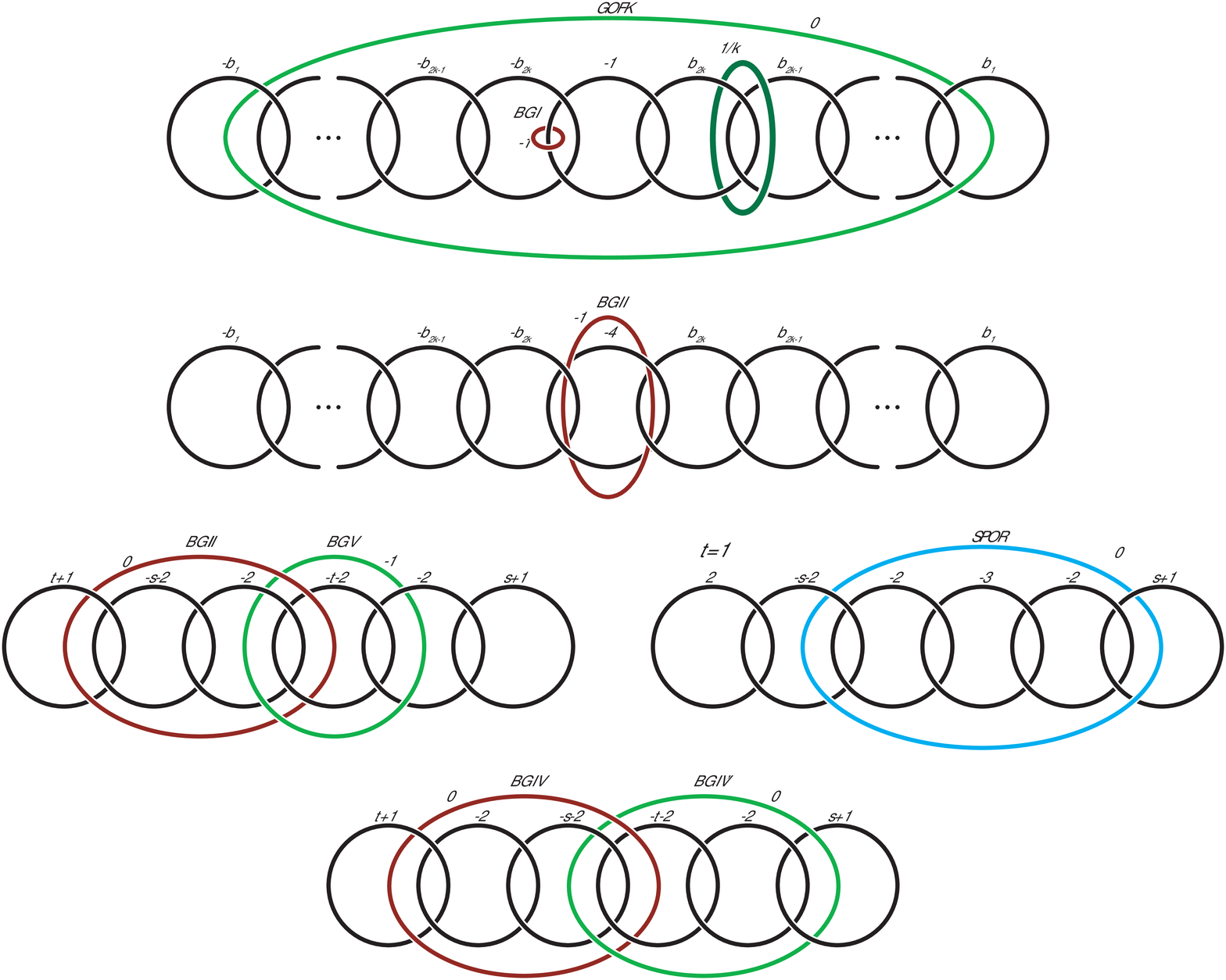}
\caption{Chain link surgery description of the lens space duals to $S^1 \times S^2$ doubly primitive knots.}
\label{fig:chainsurgery}
\end{figure}

\subsection{Continued fractions}
First we establish a few basic results about continued fractions.  These appear throughout the literature in various forms, but it is useful to set notation and collect them here.

Given the continued fraction $[a_1, \dots, a_k]^-$, define the numerators  and denominators of the ``forward'' convergents as follows:  
\[
\begin{array}{lll}
P_{-1} = 0 & P_0 = 1 & P_i = a_i P_{i-1} - P_{i-2} \\
Q_{-1} = -1 & Q_0 = 0 & Q_i = a_i Q_{i-1} - Q_{i-2}
\end{array}
\]
 \begin{claim}\label{claim:forwardP}
For $i=1, \dots, k$,  $\frac{P_i}{Q_i} = [a_1, \dots, a_i]^-$ and $P_{i-1} Q_i - P_i Q_{i-1}=1$.
\end{claim}
\begin{proof}
These are immediate when $i=1$, so assume they are true for continued fractions of length up to $i$. Writing 
\[ [a_1, \dots, a_{i-1}, a_{i}, a_{i+1}]^- = [a_1, \dots, a_{i-1}, a'_{i}]^- \]
where $a'_{i} = a_i - \frac{1}{a_{i+1}}$, the numerator of the forward convergent of the continued fraction on the right hand side is
\[
P'_i = a'_i P'_{i-1} - P'_{i-2} = (a_i - \frac{1}{a_{i+1}}) P_{i-1} - P_{i-2} = - \frac{1}{a_{i+1}} P_{i-1} + a_i P_{i-1} - P_{i-2} = - \frac{1}{a_{i+1}} P_{i-1} + P_i.
\]
 Similarly the denominator is $Q'_i = - \frac{1}{a_{i+1}} Q_{i-1} + Q_i$.
 Then 
 \[ \frac{P_{i+1}}{Q_{i+1}} = \frac{a_{i+1} P_i - P_{i-1}}{a_{i+1} Q_i - Q_{i-1}} =  \frac{P_i - \frac{1}{a_{i+1} }P_{i-1}}{ Q_i -\frac{1}{a_{i+1} } Q_{i-1}}  = \frac{P'_i}{Q'_i} = [a_1, \dots, a_{i+1}]^-\]
 as desired.    Also
 \begin{multline*}
  P_{i-1} Q_i - P_i Q_{i-1} = P_{i-1}(a_i Q_{i-1} - Q_{i-2}) - (a_i P_{i-1} - P_{i-2}) Q_{i-1} \\= P_{i-2} Q_{i-1} - P_{i-1}Q_{i-2} = \dots = P_{-1} Q_0 - P_0 Q_{-1} = 1. 
  \end{multline*}
 \end{proof}

 Given the continued fraction $[a_k, \dots, a_1]^-$, define the numerators  and denominators of the ``backward'' convergents as follows:  
\[
\begin{array}{lll}
p_{-1} = 0 & p_0 = 1 & p_i = a_i p_{i-1} - p_{i-2} \\
q_{-1} = -1 & q_0 = 0 & q_i = a_{i-1} q_{i-1} - q_{i-2}
\end{array}
\]
\begin{claim}
For $i=1, \dots, k$, $\frac{p_i}{q_i} = [a_i, \dots, a_1]^-$ and $q_i = p_{i-1}$.
\end{claim}
\begin{proof}
These are immediate when $i=1$, so assume they are true for continued fractions of length up to $i$.
First $q_{i+1} = a_i q_{i} - q_{i-1} = a_i p_{i-1} - p_{i-2} = p_i$.  Then 
\[ [a_{i+1}, a_i, \dots, a_1]^- = a_{i+1} - \frac{1}{[a_i, \dots, a_1]^-} = a_{i+1} - \frac{1}{p_i/p_{i-1}} = \frac{a_{i+1} p_i - p_{i-1}}{p_i} = \frac{p_{i+1}}{q_{i+1}}. \]
\end{proof}

\begin{lemma}
$\displaystyle [b_1, \dots, b_k, c, -b_k, \dots, -b_1]^- = \frac{c P_k^2}{c P_k Q_k + 1}$
\end{lemma}
\begin{proof}
Notice that for this continued fraction we have $\frac{P_i}{Q_i} = [b_1, \dots, b_i]^-$ and $\frac{p_i}{q_i} = [-b_i, \dots, -b_1]^-$ for $i=1, \dots, k$.  Using the definitions of $P_i$ and $p_i$ one can show that $p_i = (-1)^i P_i$ for $i=1, \dots, k$.  Then we have:
\begin{align*}
[b_1, \dots, b_k, c, -b_k, \dots, -b_1]^- &= [b_1, \dots, b_k, c-\frac{1}{[-b_k, \dots, -b_1]^-}]^- = [b_1, \dots, b_k, c-\frac{p_{k-1}}{p_k}]^- 
= \frac{(c-\frac{p_{k-1}}{p_k})P_k - P_{k-1}}{(c-\frac{p_{k-1}}{p_k})Q_k - Q_{k-1}}\\ 
& =  \frac{c p_k P_k - p_{k-1} P_k - p_k P_{k-1}}{c p_k Q_k - p_{k-1} Q_k - p_k Q_{k-1}}
=  \frac{(-1)^k c P_k P_k -(-1)^{k-1} P_{k-1} P_k - (-1)^k P_k P_{k-1}}{(-1)^k c P_k Q_k - (-1)^{k-1} P_{k-1} Q_k - (-1)^k P_k Q_{k-1}}\\
&= \frac{(-1)^k c P_k^2}{(-1)^k c P_k Q_k + (-1)^k (P_{k-1} Q_k - P_k Q_{k-1})} = \frac{c P_k^2}{c P_k Q_k + 1}
\end{align*}
\end{proof}

\subsection{Homology classes of the duals to the {\sc bg}, {\sc gofk}, and {\sc spor} knots}

We now calculate the homology classes of the knots indicated in Figure~\ref{fig:chainsurgery}.    To do so, orient and index each linear chain link $L = L_1 \cup \dots \cup L_n$ of $n$ components from right to left as in Figure~\ref{fig:chainhomology}.  Denote the exterior of this link by $X(L) = S^3 - N(L)$.  Let $\{\mu_i,\lambda_i\}$ be the standard oriented meridian, longitude pair giving a basis for the homology  of the boundary of a regular neighborhood of the $i$th component, $H_1(\bdry N(L_i))$. Then $H_1(X(L)) = \langle \mu_1, \dots, \mu_n\rangle \cong \Z^n$. Take $\lambda_i$ so that it is represented by the boundary of a meridional disk in $H_1(S^3-N(L_i))$.   Then in $H_1(X(L))$ we have $\lambda_i = \mu_{i-1} + \mu_{i+1}$ for $i \in \{1, \dots, n\}$ where $\mu_0 = \mu_{n+1} = 0$.    Let $L(-a_1, \dots, -a_n)$ denote the lens space obtained by this surgery description on the  chain link $L$ with $-a_i$ surgery on the $i$th component.  The surgery induces the relation $\lambda_i = a_i \mu_i$ for each $i$ and hence the relation $0= \mu_{i-1}- a_i \mu_i + \mu_{i+1}$ in $H_1(X(L))$.  Thus $H_1(L(-a_1, \dots, -a_n)) = \langle \mu_1, \dots, \mu_n \quad \colon \quad \mu_{i-1}- a_i \mu_i + \mu_{i+1}, i \in \{1, \dots, n\}\rangle$.

\begin{lemma}\label{lem:homology}
Let $L(p,q) = L(-a_1, \dots, -a_n)$ be the lens space described by surgery on the $n$ component chain link with surgery coefficient $-a_i$ on the $i$th component so that $\frac{p}{q} = [a_1, \dots, a_n]^-$.  

Then $\mu_{i} = P_{i-1} \mu_1$ in $H_1(L(p,q))$ for each $i=1, \dots, n$.
In particular, $p=P_n$ and $q^{-1}=P_{n-1}$ so that $q \mu_n=\mu_1$.
\end{lemma}

\begin{proof}
Since $\mu_{2} = a_1 \mu_1 = P_1 \mu_1$ and $\mu_{i+1}=  a_i \mu_i - \mu_{i-1}$, the result follows from the definition of $P_i$ and a simple induction argument.  Assuming this statement is true up through $i$,
\[\mu_{i+1}=  a_i \mu_i - \mu_{i-1} = a_i P_{i-1} \mu_1 - P_{i-2} \mu_1= (a_i P_{i-1} - P_{i-2})\mu_1 = P_i \mu_1.\]
The last statement follows since $P_{i-1} = Q_i^{-1}$ mod $P_i$ by Claim~\ref{claim:forwardP}. 
\end{proof}

\begin{proof}[Proof of Theorem~\ref{thm:mainhomology}]
Figure~\ref{fig:chainsurgery} shows linear chain link surgery descriptions of Lisca's lens spaces with additional unknotted components that describe knots in these lens spaces.
Orient these knots in Figure~\ref{fig:chainsurgery} counter-clockwise.  The homology class of each such knot $K$ may be determined in terms of the meridians of the chain link by counting $\mu_i$ for each time $K$ runs under the $i$th component to the left and counting $-\mu_i$ for each time $K$ runs under the $i$th component to the right.
Applying Lemma~\ref{lem:homology} allows us to write the homology class of $K$ in terms of $\mu_1$.  For the four families of Theorem~\ref{thm:main}, and using its notation, we obtain the following: 

\begin{itemize}

\item[(1)] With $(-a_1, \dots, -a_{4k+1}) = (-b_1, \dots,  -b_{2k}, -1, b_{2k}, \dots, b_{1})$, $\frac{p}{q} = [b_1, \dots, b_{2k}, 1, -b_{2k}, \dots, b_1]^- = \frac{m^2}{q}$ where $m = P_{2k}$, $d=Q_{2k}$ and $q = P_{2k} Q_{2k} + 1 = md+1$.  This gives both that $1-q = -dm$ and that $qm = m \mod p$.  Furthermore $\mu_1 = q \mu_{4k+1}$.

\begin{itemize}
\item[] $[K_\bgi] = -\mu_{2k+1} = -P_{2k} \mu_{1}=-m \mu_{1}=-qm\mu_{4k+1} = -m \mu_{4k+1}$  
\item[] $[K_\gofk] = \mu_1 - \mu_{4k+1} = (1-q) \mu_{1} = -dm \mu_1 = -dqm \mu_{4k+1} = -dm \mu_{4k+1}$
\end{itemize}

\item[(2)] With $(-a_1, \dots, -a_{2k+1}) = (-b_{1}, -b_{2}, \dots, -b_{k}, -4, b_k, \dots, b_k)$, $\frac{p}{q} =[b_1, \dots, b_k, 4, -b_k, \dots, -b_1] = \frac{m^2}{q}$ where $m = 2P_k$, $d=2Q_k$, and $q = 4 P_k Q_k +1=md+1$.
This gives that $qm = m \mod p$.  Furthermore $\mu_1 = q \mu_{4k+1}$.

\begin{itemize}
\item[] $[K_\bgii] =  \mu_k - 2\mu_{k+1} + \mu_{k+2} = (P_{k-1} - 2 P_{k} + P_{k+1}) \mu_{1} = (P_{k-1} - 2 P_{k} + (4 P_k - P_{k-1})) \mu_{1}  = 2 P_{k} \mu_{1} = m \mu_{1}=qm\mu_{4k+1} = m \mu_{4k+1}$ 
\end{itemize}

\item[(3)] 
With $(-a_1, \dots, -a_6) = (t + 1, -s - 2, -2, -t - 2, -2, s + 1)$, $\frac{p}{q} =[-t-1, s+2, 2, t+2,2, -s-1]^-= \frac{(4 + 3 t + 2 s  + 2 s t)^2 }{ -(3 + 2 s)^2 (1 + t)}$.  Then take $m=4+3t+2s+2st$ so that $p=m^2$ and $d =-(3+2s)$ so that  $q=d(m-1)$.   This gives that $d$ is odd,  that $(m-1)=-d(1+t)$ from which $d^{-1}m =  (1+t)m \mod p$, and that $qm = -dm \mod p$.  Furthermore $\mu_1 = q \mu_6$.

\begin{itemize}
\item[] $[K_\bgiii]= \mu_1 + \mu_4 = (1+P_3) \mu_1 =-(4 + 3 t + 2 s  + 2 s t) \mu_1=-m\mu_1 = -qm \mu_6 = dm \mu_6$ and
\item[] $[K_\bgv]= \mu_3 - \mu_5 = (P_2 - P_4) \mu_1 = (1 + t) (4 + 3 t + 2 s + 2 s t) \mu_1 =  (1 + t) m \mu_1  = d^{-1}m \mu_1 = d^{-1}qm \mu_6 = -m \mu_6$
\end{itemize}

When $t=1$ we have a third knot.  Then $(-a_1, \dots, -a_6) = (2, -s - 2, -2, -3, -2, s + 1)$, $\frac{p}{q} =[-2, s+2, 2, 3, 2, -s-1]^-= \frac{(7+ 4 s)^2 }{ -2(3 + 2 s)^2 }$, and $m=1-2d$.

\begin{itemize}
\item[] $[K_\bgiii]=-(7+4s)\mu_1=-m \mu_1 = dm \mu_6$,
\item[] $[K_\bgv]=2(7+4s) \mu_1 = 2m\mu_1=-2dm \mu_6 = (m-1)m\mu_6 = -m\mu_6$, and
\item[] $[K_\spor]= \mu_2 - \mu_6 = (P_1-P_5)\mu_1=4(7+4s) \mu_1= 4m \mu_1=-2m\mu_6$. 
\end{itemize}

\item[(4)] With $(-a_1, \dots, -a_6) = (t + 1, -2, -s - 2, -t - 2, -2, s + 1)$, $\frac{p}{q} = [-t-1,2,s+2,t+2,2,-s-1]^-=\frac{(4 + 3s + 3 t + 2 s t)^2}{-(3 + 2 s) (3 + 3 s + 3 t + 2 s t)}$.  Then take $m=4 + 3s + 3 t + 2 s t$ so that $p=m^2$ and $d=-(3+2s)$ so that $q=d(m-1)$.  This gives both that $(2m+1) = -d(3+2t)$ from which $d^{-1}m = -(3+2t)m \mod p$ and that $qm=-dm \mod p$.  Furthermore $\mu_1 = q \mu_6$.

\begin{itemize}
\item[] $[K_\bgiv]= \mu_1 + \mu_4=(1+P_3)\mu_1=- (4 + 3s + 3 t + 2 s t)\mu_1=-m\mu_1 = -qm \mu_6 = dm \mu_6$
\item[] $[K'_\bgiv]= \mu_3 + \mu_6=(P_2+P_5) \mu_1 =-(3 + 2 t) (4 + 3s + 3 t + 2 s t) \mu_1=d^{-1}m\mu_1= d^{-1}qm \mu_6 = -m \mu_6$.
\end{itemize}

 \end{itemize}

Let $\mu$ and $\mu'$ be the homology classes of the two cores of the Heegaard solid tori of $L(p,q)$ suitably oriented so that $q\mu = \mu'$. Then, we have $\mu = \mu_{4k+1}$ for (1) and (2) and $\mu = \mu_6$ for (3) and (4) in the calculations above.  Since a knot's orientation does not effect its Dehn surgeries, taking both signs of the homology classes above competes the proof of Theorem~\ref{thm:mainhomology}.
\end{proof}

\begin{figure}
\centering
\includegraphics[width=5in]{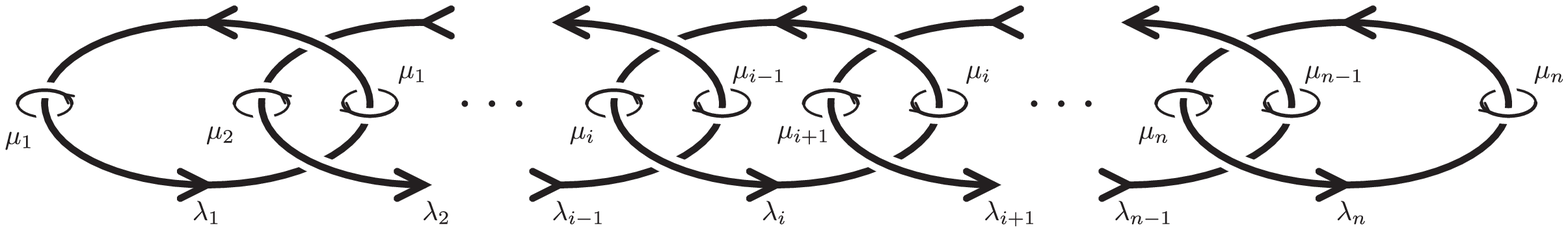}
\caption{}
\label{fig:chainhomology}
\end{figure}

\begin{lemma}\label{lem:spor}
The {\sc spor} knots generically are not Berge-Gabai knots.  

In particular, when $t=1$ and $n=s+2 \neq -1, 0$ the surgery duals to {\sc spor}, {\sc bgiii}, {\sc bgv} are mutually distinct.  When $t=1$ and $n=s+2 = 0$, these knots are all the unknot in $S^3$.  When $t=1$ and $n=s+2 = -1$ the knots {\sc spor} and {\sc bgiii} are isotopic but distinct from {\sc bgv}. 
\end{lemma}

\begin{proof}
Up to mirroring, the lens space obtained by longitudinal surgery on a sporadic knot is $L(p,q) =L((7+ 4 s)^2 , -2(3 + 2 s)^2)$.  Let us reparametrize by $s=n-2$ so that $L(p,q) = L((4n-1)^2, -2(2n-1)^2) = L((4n-1)^2, 8n^2-1)$.  Again, we take $\mu$ and $\mu'$ to be the homology classes of the oriented cores of the Heegaard solid tori so that $q\mu=\mu'$. Then by Theorem~\ref{thm:mainhomology} the unoriented knot dual to the sporadic knot represents the homology classes  $\pm2(4n-1)\mu$ while the duals to the \bgiii\ and \bgv\ knots in this lens space represent the homology classes $\pm(2n-1)(4n-1) \mu$ and $\pm(4n-1)\mu$.
Since $q^2 \not \equiv \pm1 \pmod p$ for $n \neq -1, 0$, the group of isotopy classes of diffeomorphisms of our lens space is $\Z/2\Z$, generated by the involution whose quotient is the two bridge link, \cite{diffeolenticulaires,diffeooflensspace}.  This involution acts on $H_1(L(p,q))$ as multiplication by $-1$.  Therefore when $n\neq -1, 0$  (and when $s \neq -3, -2$) the duals to {\sc bgiii}, {\sc bgv}, and {\sc spor} are mutually non-isotopic.

For $n=0$, $L(p,q) =S^3$ so that the knot dual to {\sc spor}, {\sc bgiii}, and {\sc bgv} are all the unknot.

For $n=-1$, $L(p,q)=L(25,7)$ the knots dual to {\sc spor} and {\sc bgiii} represent the homology classes $\pm10\mu$.  One may directly observe that the corresponding knots are isotopic.  The knot dual to {\sc bgv} represents the homology classes $\pm5\mu$.  The knots in $S^1 \times S^2$ with integral surgeries yielding $L(25,7)$ are those shown to the left and center in Figure~\ref{fig:smallhyperbolicknots}.   
\end{proof}

\begin{cor}
The three bandings of the two-bridge links in Figure~\ref{fig:quotients} are distinct up to homeomorphism of the two-bridge link.
\end{cor}

\section{Doubly primitive knots, waves, and simple knots}
\label{sec:simplewave}

We now generalize Berge's results that the duals to doubly primitive knots in $S^3$ (under the associated lens space surgery) are simple knots and that $(1,1)$--knots with longitudinal $S^3$ surgeries are simple knots.  We will adapt Saito's proofs given in the appendix of \cite{saito}.

 A {\em wave} of a genus $2$ Heegaard diagram $(S, \bar{x}= \{x_1, x_2\}, \bar{y}=\{y_1, y_2\})$ is an arc $\alpha$ embedded in $S$ so that (up to swapping $x$'s and $y$'s) $\alpha \cap \bar{x} = \bdry \alpha \subset x_i$ for  $i=1$ or $2$, at each endpoint $\alpha$ encounters $x_i$ from the same side, and each component of $x_i - \alpha$ intersects $\bar{y}$.  A regular neighborhood of $\alpha \cup x_i$ is a thrice-punctured sphere of which one boundary component is not isotopic to a member of $\bar{x}$.  A {\em wave move} along $\alpha$ is the replacement of $x_i$ by this component.

Let us say two simple closed curves on an orientable surface {\em coherently intersect} if they may be oriented so that every intersection occurs with the same sign. (This includes the possibility that the two curves are disjoint.) We then say a Heegaard diagram is {\em coherent} if every pair of curves in the diagram coherently intersect.  

Say a $3$--manifold $W$ of Heegaard genus at most $2$ is {\em wave-coherent} if any genus $2$ Heegaard diagram $(S, \bar{x}= \{x_1, x_2\}, \bar{y}=\{y_1, y_2\})$ of $W$ either admits a wave move or  is coherent.

\begin{thm}\label{thm:wavesimplesurgery}\
\begin{enumerate}
\item If longitudinal surgery on a $(1,1)$--knot in a lens space  produces a wave-coherent manifold, then the knot is simple.
\item Given a doubly primitive knot in a wave-coherent manifold of Heegaard genus at most $2$, the surgery dual to the associated lens space surgery is a simple knot.
\end{enumerate}
\end{thm}

\begin{proof}
The proof of the first follows exactly the same as that of Saito's Theorem~A.5 (with Lemma~A.6) in \cite{saito} except that we use Proposition~\ref{prop:coherentintersect} below in the stead of his Proposition~A.1.

The second item then follows because the surgery dual to a doubly primitive knot is a $(1,1)$--knot.  See Theorem~A.4 \cite{saito} for example.
\end{proof}

\begin{cor}
A $3$--manifold of genus at most $2$ obtained by longitudinal surgery on a non-trivial $(1,1)$--knot in $S^3$ or $S^1 \times S^2$ is not wave-coherent.
\end{cor}

\begin{proof} 
Theorem~\ref{thm:wavesimplesurgery} applies even if the $(1,1)$--knot is in $S^3$ or in $S^1 \times S^2$.  The trivial knot is the only simple knot in these two manifolds.
\end{proof}

\begin{proof}[Proof of Theorem~\ref{thm:simplewave}]
$S^1 \times S^2$ is wave-coherent by Theorem~\ref{thm:negamiokita} (\cite{negamiokita}) so the result follows from Theorem~\ref{thm:wavesimplesurgery}.
\end{proof}

\begin{prop}[Cf.\ Proposition A.1 \cite{saito}]\label{prop:coherentintersect}
Let $(S; \bar{x}= \{x_1, x_2\}, \bar{y}=\{y_1, y_2\})$ be a normalized Heegaard diagram of a $3$--manifold $W$.
Assume $z$ is a simple closed curve in $S$ such that $z$ intersects each $x_1$ and $y_1$ once and is disjoint from both $x_2$ and $y_2$.
If $W$ is wave-coherent, then $x_2$ and $y_2$ coherently intersect.
\end{prop}

\begin{proof}[Sketch of Proof]
Saito's proof of the analogous theorem for $W=S^3$ applies to any wave-coherent manifold of genus at most $2$ whose genus $2$ Heegaard diagrams enjoy the NEI Property:
A Heegaard diagram $(S;\bar{x},\bar{y})$ is said to have the {\em Non-Empty Intersecting (NEI) Property} if every $x_i \in \bar{x}$ intersects some $y_j \in \bar{y}$ and every $y_j \in \bar{y}$ intersects some $x_i \in \bar{x}$. 
Any genus $2$ Heegaard diagram for $S^3$ (or any homology sphere) enjoys the NEI Property by \cite[Lemma 1]{ochiai}, and $S^3$ is wave-coherent by \cite{hot}.  The main tool is Ochiai's structure theorem for Whitehead graphs of genus $2$ Heegaard diagrams with the NEI Property, \cite[Theorem 1]{ochiai}.

Assume $(S, \bar{x}, \bar{y})$ does not enjoy the NEI Property.
Then the manifold $W$ contains a non-separating sphere and hence an $S^1 \times S^2$ summand.  It follows that $W$ is homeomorphic to $S^1\times S^2 \# L(p,q)$ for some integer $p$.  All such manifolds are all wave-coherent by \cite[Theorem 1-4]{negamiokita}.

If $(S, \bar{x}, \bar{y})$ is a standard Heegaard diagram for $W\cong S^1\times S^2 \# L(p,q)$, then it is simple and the proposition is satisfied, so further assume the diagram is not standard.  Assume $y_0 \in \{y_1,y_2\}$ does not intersect $x_1 \cup x_2$. Then since the diagram is not standard, $y_0$ cannot be parallel to either $x_1$ or $x_2$.  Because $y_0$ is non-separating, $y_0 \cup x_1 \cup x_2$ must be the boundary of thrice-punctured sphere in $S$.  Since $z$ intersects $x_1$ just once and is disjoint from $x_2$, it must also intersect $y_0$.  Therefore $y_0=y_1$.  Hence the Heegaard diagram with $z$ must appear as in Figure~\ref{fig:noNEIdiagram} after gluing $x_i^+$ to $x_i^-$ for each $i=1,2$ to reform $S$.  The thick arcs labeled $a$ and $b$ represent sets of $a$ or $b$ parallel arcs of $y_2 - (x_1 \cup x_2)$.  Because $z$ intersects $x_1$ once, it dictates how the ends of the rest of the arcs encountering $x_1$ must match up.  Since these other arcs all together constitute the single curve $y_2$, we must have either $b=1$ and $a=0$ or $b=0$ and $a>0$.  In either case the conclusion of the proposition holds.
\end{proof}

\begin{figure}
\centering
\includegraphics[height=1.5in]{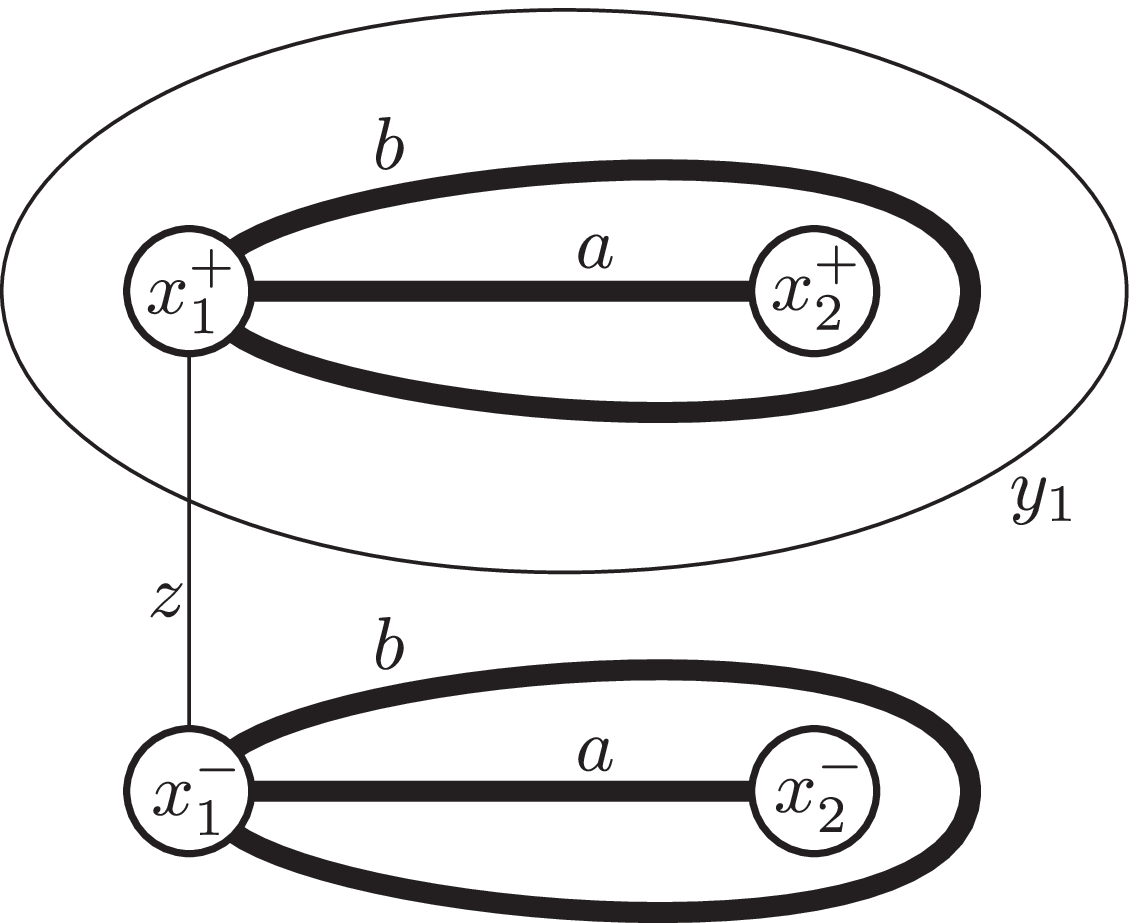}
\caption{}
\label{fig:noNEIdiagram}
\end{figure}

\begin{question}
Which $3$--manifolds are wave-coherent? 
Homma-Ochiai-Takahashi show $S^3$ is wave-coherent \cite{hot}, and Negami-Ochiai show the manifolds $S^1 \times S^2 \# L(p,q)$ are wave-coherent \cite{negamiokita}. In each of these cases, wave moves reduce genus $2$ Heegaard diagrams into a standard one.   On the other hand, note that Osborne shows the lens spaces $L(173,78)$ and $L(85,32)$ admit genus $2$ diagrams with fewer crossings than the standard stabilization of a genus $1$ diagram \cite{osborne}, and hence wave moves alone will not necessarily transform any genus $2$ diagram of these lens spaces into the standard stabilized diagram.  Nevertheless these minimal diagrams of Osborne are coherent.   Are these lens spaces wave-coherent?
\end{question}

\subsection{On the classification of doubly primitive knots in $S^1 \times S^2$}

By Theorem~\ref{thm:simplewave}, the surgery dual to a doubly primitive knot in $S^1 \times S^2$ is a simple knot.  To prove our families {\sc bg}, {\sc gofk}, and {\sc spor} constitute all doubly primitive knots in $S^1 \times S^2$, it remains to show that no simple knot in Lisca's lens spaces other than those in the homology classes of Theorem~\ref{thm:mainhomology} admit a surgery to $S^1 \times S^2$.   

One approach is to (a) show that the simple knots of the correct homological order in Lisca's lens spaces have fibered exterior and then (b) determine which of these have planar fibers.  Cebanu has confirmed the part (a) employing theorems of Brown \cite{brown} and Stallings \cite{stallings} in the vein of Ozsvath-Szabo's proof that Berge's doubly primitive knots are fibered \cite{OSLspace}.  As of this writing, Cebanu has completed part (b) for the first two types of Lisca's lens spaces $L(m^2, md\pm1)$ where $\gcd(m,d)=1$ or $2$ \cite{cebanu}.

One may care to consider alternative approaches of considering either  the fundamental group of the result of the homologically correct surgery on the simple knots (see e.g.~\cite{tange}) or the bandings of the associated tangles.

\section{Knots in lens spaces with $S^1 \times S^2$ surgeries from lattice embeddings}
\label{sec:embeddings}

Recall from section~\ref{sec:notation} that the expression $\frac{p}{q}=[a_1,...,a_n]^-$ with $a_i>2$ induces the Kirby diagram of Figure~\ref{fig:general} for a negative definite plumbing manifold $P(p,q)$ whose boundary is the lens space $L(p,q)$, and that the $\pi$--rotation $u$ in the diagram describes an involution  that expresses $L(p,q)$ as the double cover of $S^3$ branched over the two-bridge link $K(p,q)$.  See also Figures~\ref{fig:chainto4plat} and \ref{fig:chainto4platv2}.

\begin{figure}
\centering
\executeiffilenewer{general.svg}{general.eps}%
{inkscape -z -D --file=general.svg %
--export-eps=general.eps --export-latex}%
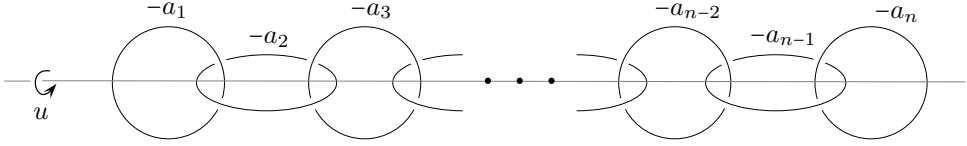%

\caption{A Kirby diagram of the $4$--manifold $P(p,q)$ with lens space boundary $L(p,q)$.}
\label{fig:general}
\end{figure}
 
In his work \cite{lisca} confirming the slice-ribbon conjecture for two-bridge knots, Lisca observes the following:  
Assuming $L(p,q)$ smoothly bounds a rational homology ball $W$, then $Z = P(p,q) \cup -W$ is a smooth, closed, negative definite $4$--manifold with $b_2(Z) = b_2(P(p,q))=n$.  Then by Donaldson's celebrated theorem, the intersection pairing on $H_2(Z;\Z)$ is isomorphic to $-\Id$. Calling $Q_{p,q}$ the intersection pairing  of $P(p,q)$, it follows that the lattice $(\Z^n,Q_{p,q})$ must embed in the standard negative definite intersection lattice of equal rank; that is, there must exist a monomorphism $\phi \colon \Z^n\rightarrow\Z^n$ such that $Q_{p,q}(\alpha,\beta)=-\Id(\phi (\alpha),\phi (\beta))$ for all $\alpha,\beta\in\Z^n\simeq H_2(P(p,q);\Z)/\rm{Tors}$.

In determining which lens spaces $L(p,q)$ bound rational homology balls, Lisca determines which of these lattices $(\Z^n,Q_{p,q})$ admit such an embedding in terms of the coefficients in a continued fraction expansion of $\frac{p}{q}$ by explicitly describing an embedding.   Moreover, these embeddings are essentially unique as discussed in Lemma~\ref{lem:uniqueembedding}. 

In light of how Greene's work \cite{greene} on embeddings  of co-rank 1 lattices with an orthogonal ``changemaker'' vector yields a classification of the pairs $(L,\kappa)$ of a lens space $L$ and homology class $\kappa \in H_1(L)$ of the surgery duals to knots in $S^3$, in this section we examine Lisca's lattice embeddings and determine how they may yield information about the knots dual to these surgeries. 
(Cf.\ Remark~\ref{rem:embeddings}.) 
We observe that these embeddings suggest a method for finding knots in the corresponding lens spaces that admit longitudinal Dehn surgeries to $S^1 \times S^2$.  
While we do not yet have a formal framework for this construction, the knots we obtain through this method are precisely (the duals to) the Berge-Gabai knots and the {\sc gofk} knots.  Curiously, the duals to the {\sc spor} knots do not fall out so directly, though knowing they exist we may locate them.

In section~\ref{sec:liscasembeddings} below we first review Lisca's classification of which lens spaces bound rational homology balls and make explicit the corresponding lattice embeddings.  Then in section~\ref{sec:latticeknots} we describe our procedure for obtaining knots in these lens spaces with longitudinal $S^1 \times S^2$ surgeries and state the results of its application to all of Lisca's embeddings.  We then demonstrate this procedure in section~\ref{sec:explicitexample} with a key example that allows us to compare how the procedure yields {\sc bgiii} and {\sc bgv} knots while stops short of yielding the {\sc spor} knots.

\subsection{Lisca's lens spaces and embeddings of lattices}
\label{sec:liscasembeddings}
We shall use the notational shortcut
\[(\dots,2^{[t]},\dots):=(\dots,\overbrace{2,\dots,2}^t,\dots)\]

\begin{lemma}[Lisca, \cite{lisca} Lemmas 7.1, 7.2, 7.3]\label{l:lisca}
If $L(p,q)$ bounds a rational homology ball then $\frac{p}{q}=[a_1,\dots,a_n]^-$ or $\frac{p}{q}=[a_n,\dots,a_1]^-$ where the string of integers $(-a_1,\dots,-a_n)$ is of one of the following types\footnote{In \cite[Lemma~7.2]{lisca} family $(4)$ is missing from the statement. See Remark~\ref{rem:missing}.}:
\[
\begin{array}{rll}
(1) & (-b_k,-b_{k-1},\dots,-b_1,-2,-c_1,\dots,-c_{l-1},-c_\ell), & k,\ell\geq 1, \\
(2)&(-2^{[t]},-3,-2-s,-2-t,-3,-2^{[s]}),& s,t\geq 0,\\
(3)&(-2^{[t]},-3-s,-2,-2-t,-3,-2^{[s]}),& s,t\geq 0,\\
(4)&(-b_k,-b_{k-1},\dots,-b_1-1,-2,-2,-1-c_1,\dots,-c_{l-1},-c_\ell),& k,\ell\geq 1,\\
(5)&(-t-2,-s-2,-3,-2^{[t]},-4,-2^{[s]}),&s,t\geq 0,\\
(6)&(-t-2,-2,-3-s,-2^{[t]},-4,-2^{[s]}),& s,t\geq 0,\\
(7)&(-t-3,-2,-3-s,-3,-2^{[t]},-3,-2^{[s]}),&s,t\geq 0.
\end{array}
\]
where the $k$--tuple of integers $b_1,\dots,b_k\geq 2$ is arbitrary and the numbers $c_1,\dots,c_\ell\geq 2$ are such that if $\frac{p}{q}=[b_1,\dots,b_k]^-$ and $\frac{r}{s}=[c_1,\dots,c_\ell]^-$ then $\frac{q}{p}+\frac{s}{r}=1$.
\end{lemma}

In order to make explicit the embeddings $\phi$ of the lattices defined by the intersection pairings associated to the lens spaces in Lemma~\ref{l:lisca} we need to introduce some notation. Let $E=\{ e_1, \dots, e_{n}\}$ be a basis of the negative diagonal lattice (so that $-\mbox{Id}(e_i,e_j)=e_i \cdot e_j = -\delta_{ij}$) and let $V=\{v_1, \dots, v_n\}$ denote the standard basis of $H_2(P(p,q))$ given by the Kirby diagram in Figure~\ref{fig:general}. We know, from Lisca's work, that if $\phi(v_i)=\sum_{j=1}^n\lambda_je_j$ then $|\lambda_j|\leq1$, and therefore we can summarize $\phi$ as in the following tables, where the signs $+$ and $-$ stand for $\lambda_j=+ 1$ and $\lambda_j= -1$ respectively, a blank stands for $\lambda_j=0$, and the number in the top left corner refers to the numbering of types in Lemma~\ref{l:lisca}. Note that the following holds:
\begin{enumerate}
\item $Q_{p,q}(v_i,v_i)=\phi(v_i)\cdot\phi(v_i)=-a_i$, 
\item $Q_{p,q}(v_i,v_j)=\phi(v_i)\cdot\phi(v_j)=1$ if $|i-j|=1$, and
\item $Q_{p,q}(v_i,v_j)=\phi(v_i)\cdot\phi(v_j)=0$ if $|i-j|>1$.
\end{enumerate}	
Whenever the meaning is clear we will drop the $\phi$ from the notation and write $v_i=\sum_{j=1}^n\lambda_je_j$ for $\phi(v_i)=\sum_{j=1}^n\lambda_je_j$.  First we give embeddings for types  (2), (3), (5), (6), and (7) and then we give embeddings for types (1) and (4) as these latter two take on a different character.  These are all implicit in \cite[Section 7]{lisca}.

{\small
\[
\begin{array}{r|c@{\hspace{0.5mm}}c@{\hspace{0.5mm}}c@{\hspace{0.5mm}}c@{\hspace{0.5mm}}c@{\hspace{0.5mm}}c@{\hspace{0.5mm}}c@{\hspace{0.5mm}}c@{\hspace{0.5mm} }c@{\hspace{0.5mm}}c@{\hspace{0.5mm}}c@{\hspace{0.5mm}}c@{\hspace{0.5mm}}c@{\hspace{0.5mm}}c}
(2)     &e_1&e_2&e_3&e_4&e_5&e_6   &\cdots&e_{t+3}&e_{t+4}&e_{t+5}&e_{t+6}&\cdots&e_{t+s+3}&e_{t+s+4}\\
\hline
v_1     &   &   &   &   &   &      &       & -     &  +    &       &       &      &         &         \\
\vdots  &   &   &   &   &   &      &\udots &       &       &       &       &      &         &         \\
v_{t-1} &   &   &   &   & - & +    &       &       &       &       &       &      &         &         \\
v_t     &   &   & - &   & + &      &       &       &       &       &       &      &         &         \\
v_{t+1} &   & + & + & + &   &      &       &       &       &       &       &      &         &         \\
v_{t+2} & + & - &   &   &   &      &       &       &       & -     &   -   &\cdots& -       &  -      \\
v_{t+3} &   & + & - &   & - &-     &\cdots & -     &     - &       &       &      &         &         \\
v_{t+4} & - & - &   & + &   &      &       &       &       &       &       &      &         &         \\
v_{t+5} & + &   &   &   &   &      &       &       &       & +     &       &      &         &         \\
v_{t+6} &   &   &   &   &   &      &       &       &       &  -    & +     &      &         &         \\
\vdots  &   &   &   &   &   &      &       &       &       &       &       &\ddots&         &         \\
v_{t+s+4}&  &   &   &   &   &      &       &       &       &       &       &      &     -   &  +      \\
\end{array}
\]
}
{\small
\[
\begin{array}{r|c@{\hspace{0.5mm}}c@{\hspace{0.5mm}}c@{\hspace{0.5mm}}c@{\hspace{0.5mm}}c@{\hspace{0.5mm}}c@{\hspace{0.5mm}}c@{\hspace{0.5mm}}c@{\hspace{0.5mm} }c@{\hspace{0.5mm}}c@{\hspace{0.5mm}}c@{\hspace{0.5mm}}c@{\hspace{0.5mm}}c@{\hspace{0.5mm}}c}
(3)     &e_1&e_2&e_3&e_4&e_5&e_{6} &\cdots&e_{t+3}&e_{t+4}&e_{t+5}&e_{t+6}&\cdots&e_{t+s+3}&e_{t+s+4}\\
\hline
v_1     &   &   &   &   &   &      &       & -     &  +    &       &       &      &         &         \\
\vdots  &   &   &   &   &   &      &\udots &       &       &       &       &      &         &         \\
v_{t-1}   &   &   &   &   & - & +    &       &      &       &       &       &      &         &         \\
v_t     &   &   & - &   & + &      &       &       &       &       &       &      &         &         \\
v_{t+1} &   & + & + & + &   &      &       &       &       & +     & +     &\cdots& +       &  +      \\
v_{t+2} & + & - &   &   &   &      &       &       &       &       &       &      &         &         \\
v_{t+3} &   & + & - &   & - &-     &\cdots & -     &     - &       &       &      &         &         \\
v_{t+4} & - & - &   & + &   &      &       &       &       &       &       &      &         &         \\
v_{t+5} &   &   &   & - &   &      &       &       &       & +     &       &      &         &         \\
v_{t+6} &   &   &   &   &   &      &       &       &       &  -    & +     &      &         &         \\
\vdots  &   &   &   &   &   &      &       &       &       &       &       &\ddots&         &         \\
v_{t+s+4}&  &   &   &   &   &      &       &       &       &       &       &      &     -   &  +      \\
\end{array}
\]
}

{\small
\[
\begin{array}{r|c@{\hspace{0.5mm}}c@{\hspace{0.5mm}}c@{\hspace{0.5mm}}c@{\hspace{0.5mm}}c@{\hspace{0.5mm}}c@{\hspace{0.5mm}}c@{\hspace{0.5mm}}c@{\hspace{0.5mm} }c@{\hspace{0.5mm}}c@{\hspace{0.5mm}}c@{\hspace{0.5mm}}c@{\hspace{0.5mm}}c@{\hspace{0.5mm}}c}
(5)     &e_1&e_2&e_3&e_4&e_5&e_6&\cdots&e_{t+3}&e_{t+4}&e_{t+5}&e_{t+6}&\cdots&e_{t+s+3}&e_{t+s+4}\\
\hline
v_1     & + & - &   &   & + &   &\cdots& +     &  +    &       &       &      &         &         \\
v_2     &   & + & - &   &   &   &      &       &       &   -   &  -    &\cdots&  -      &  -      \\
v_3     & - & - &   & + &   &   &      &       &       &       &       &      &         &         \\
v_4     & + &   &   &   & - &   &      &       &       &       &       &      &         &         \\
v_5     &   &   &   &   & + & - &      &       &       &       &       &      &         &         \\
\vdots  &   &   &   &   &   &   &\ddots&       &       &       &       &      &         &         \\
v_{t+3} &   &   &   &   &   &   &      & +     & -     &       &       &      &         &         \\
v_{t+4} &   & + & + & + &   &   &      &       &  +    &       &       &      &         &         \\
v_{t+5} &   &   & - &   &   &   &      &       &       & +     &       &      &         &         \\
v_{t+6} &   &   &   &   &   &   &      &       &       & -     &+      &      &         &         \\
\vdots  &   &   &   &   &   &   &      &       &       &       &       &\ddots&         &         \\
v_{t+s+4}&  &   &   &   &   &   &      &       &       &       &       &      &     -   &  +      \\
\end{array}
\]
}

{\small
\[
\begin{array}{r|c@{\hspace{0.5mm}}c@{\hspace{0.5mm}}c@{\hspace{0.5mm}}c@{\hspace{0.5mm}}c@{\hspace{0.5mm}}c@{\hspace{0.5mm}}c@{\hspace{0.5mm}}c@{\hspace{0.5mm} }c@{\hspace{0.5mm}}c@{\hspace{0.5mm}}c@{\hspace{0.5mm}}c@{\hspace{0.5mm}}c@{\hspace{0.5mm}}c}
(6)     &e_1&e_2&e_3&e_4&e_5&e_6&\cdots&e_{t+3}&e_{t+4}&e_{t+5}&e_{t+6}&\cdots&e_{t+s+3}&e_{t+s+4}\\
\hline
v_1     & + & - &   &   & + &   &\cdots& +     &  +    &       &       &      &         &         \\
v_2     &   & + & - &   &   &   &      &       &       &       &       &      &         &         \\
v_3     & - & - &   & + &   &   &      &       &       &   +   &  +    &\cdots&  +      &  +      \\
v_4     & + &   &   &   & - &   &      &       &       &       &       &      &         &         \\
v_5     &   &   &   &   & + & - &      &       &       &       &       &      &         &         \\
\vdots  &   &   &   &   &   &   &\ddots&       &       &       &       &      &         &         \\
v_{t+3} &   &   &   &   &   &   &      & +     & -     &       &       &      &         &         \\
v_{t+4} &   & + & + & + &   &   &      &       &  +    &       &       &      &         &         \\
v_{t+5} &   &   &   & - &   &   &      &       &       & +     &       &      &         &         \\
v_{t+6} &   &   &   &   &   &   &      &       &       & -     &+      &      &         &         \\
\vdots  &   &   &   &   &   &   &      &       &       &       &       &\ddots&         &         \\
v_{t+s+4}&  &   &   &   &   &   &      &       &       &       &       &      &     -   &  +      \\
\end{array}
\]
}

{\small
\[
\begin{array}{r|c@{\hspace{0.5mm}}c@{\hspace{0.5mm}}c@{\hspace{0.5mm}}c@{\hspace{0.5mm}}c@{\hspace{0.5mm}}c@{\hspace{0.5mm}}c@{\hspace{0.5mm}}c@{\hspace{0.5mm} }c@{\hspace{0.5mm}}c@{\hspace{0.5mm}}c@{\hspace{0.5mm}}c@{\hspace{0.5mm}}c@{\hspace{0.5mm}}c@{\hspace{0.5mm}}c}
(7)     &e_1&e_2&e_3&e_4&e_5&e_6&e_7   &\cdots&e_{t+4}&e_{t+5}&e_{t+6}&e_{t+7}&\cdots&e_{t+s+4}&e_{t+s+5}\\
\hline
v_1     & + & + & + &   &   & + & +    & \cdots&   +   & +     &       &      &      &         &         \\
v_2     &   &   & - & + &   &   &      &       &       &       &       &      &      &         &         \\
v_3     &   & - & + &   &+  &   &      &       &       &       &  +    &+     &\cdots&  +      &   +     \\
v_4     & + &   & - & - &   &   &      &       &       &       &       &      &      &         &         \\
v_5     & - &   &   &   &   & + &      &       &       &       &       &      &      &         &         \\
v_6     &   &   &   &   &   & - & +    &       &       &       &       &      &      &         &         \\
\vdots  &   &   &   &   &   &   &      &\ddots &       &       &       &      &      &         &         \\
v_{t+4} &   &   &   &   &   &   &      &       &  -    & +     &       &      &      &         &         \\
v_{t+5} &   &+  &   &   &+  &   &      &       &       & -     &       &      &      &         &         \\
v_{t+6} &   &   &   &   & - &   &      &       &       &       &+      &      &      &         &         \\
v_{t+7} &   &   &   &   &   &   &      &       &       &       &-      & +    &      &         &         \\
\vdots  &   &   &   &   &   &   &      &       &       &       &       &      &\ddots&         &         \\
v_{t+s+5}&  &   &   &   &   &   &      &       &       &       &       &      &      &  -      &   +     \\
\end{array}
\]
}

For the embeddings of types (1) and (4) in Lemma~\ref{l:lisca}, consider the following two operations preformed on a string of integers

\begin{itemize}
\item[(a)] $(-a_1,\dots,-a_n)\longrightarrow (-2,-a_1,\dots,-a_n-1)$
\item[(b)] $(-a_1,\dots,-a_n)\longrightarrow (-a_1-1,\dots,-a_n,-2).$
\end{itemize}

Type (1) strings are obtained from $(-2,-2,-2)$ by performing a sequence of operations (a) and (b).   Type (4) strings are obtained analogously from the string $(-3,-2,-2,-3)$. 
Embeddings of these two strings are give below.
{\small
\[
\begin{array}{r|c@{\hspace{0.5mm}}c@{\hspace{0.5mm}}c}
    &e_1&e_2&e_3\\
\hline
v_1     & + & - &         \\
v_2     &   & + & -    \\
v_3     & - & - &     \\
\end{array}
\quad \quad \quad \quad \quad \quad
%
\begin{array}{r|c@{\hspace{0.5mm}}c@{\hspace{0.5mm}}c@{\hspace{0.5mm}}c}
    &e_1&e_2&e_3&e_4\\
\hline
v_1     &  & + & + &+         \\
v_2     & + & - &  &    \\
v_3     &  & + & - &     \\
v_4     & - & - &  & +    \\
\end{array}
\]
}

Notice that in both cases there is an $e_i$ that appears only in the extremal vectors: $e_1$ in type (1) and $e_4$ in type (4).   This permits operations (a) and (b) to extend to the associated embeddings of the strings. For example, applying operation (a) to $(-2,-2,-2)$  gives $(-2,-2,-2,-3)$ and the embedding above becomes $v'_1=-e_1+e_4$, $v_2'=v_1$, $v_3'=v_2$, and $v'_4=v_3-e_4$.  In a similar fashion one may explicitly obtain the embedding of any string of type (1) or (4).

\begin{lemma}\label{lem:uniqueembedding}
The above embeddings are unique up to reindexing the basis vectors and scaling by a factor of $-1$.
\end{lemma}

\begin{proof}
This follows from the proof of \cite[Theorem~6.4]{lisca} that states that if a negative plumbing associated to a lens space admits an embedding, then this embedding can be obtained from the embedding of the lattice associated to the string $(-2,-2,-2)$ by a sequence of operations called ``expansions''. The embedding of $(-2,-2,-2)$ is unique up to reindexing and scaling, and the expansions share this property. On the other hand, the reader can easily check that whenever there is a $-2$--chain in the plumbing, that is $n\geq 1$ consecutive unknots with framing $-2$ in the diagram in Figure~\ref{fig:general}, the embedding of this chain is unique up to reindexing and scaling. Indeed, the only combination of basis vectors with square $-2$ are given by $\pm e_i\pm e_j$; and if another $-2$--framed unknot is linked to this one, then its embedding must be of the form $\mp e_i \pm e_k$ or $\pm e_k \mp e_j$. Families $(2),(3),(5),(6)$ and $(7)$ in Lemma~\ref{l:lisca} consist of $-2$--chains and at most $4$ unknots with framings different from $-2$. One can check that once the embedding of the $-2$--chains is fixed, the embedding of the rest of the unknots in the diagram is forced, and therefore the embedding is unique up to reindexing and scaling. It remains true for families $(1)$ and $(4)$ that once the embedding of the $-2$--chains is fixed the rest of the embedding is forced. However, since in these two families the number of unknots with framing different from $-2$ is arbitrary, it is more cumbersome to show it directly than to deduce it  from Lisca's general analysis on the embedding of lattices associated to lens spaces.
\end{proof}

\subsection{Surgeries from lattice embeddings}\label{sec:latticeknots}
Here we give a heuristic for finding knots in Lisca's lens spaces with $S^1 \times S^2$ surgeries from the lattice embeddings.
With the assumption that such a knot should be suitably ``simple'' in some sense, we restrict attention to $-1$ surgery on unknots in the Kirby diagram (and hence blowdowns) that are equivariant with respect to the involution and have an ``uncomplicated'' presentation in hopes that blowing down such an unknot will lead to further reductions of the Kirby diagram.

To begin such a sequence of reductions we look for a coefficient $a_j$ that equals $2$.  The corresponding component of the chain link has framing $-2$.  A blowdown along an unknot linking this component once will change its framing to $-1$, prompting a subsequent blowdown.  Given the lattice embedding, if $a_j=2$, then $Q_{p,q}(v_j,v_j)=-2$ and $\phi(v_j) = \varepsilon_k e_k + \varepsilon_\ell e_\ell$ for some choices of $k, \ell$ and signs $\varepsilon_k, \varepsilon_\ell=\pm1$.   Select either of these two basis vectors, say $e_k$.  The vectors $v_i$ that have non-trivial $e_k$ component then indicate the components of the chain link that some unknot $K$ should link so that blowing down along $K$ should initiate a chain of blowdowns yielding $S^1 \times S^2$.

Thus for each rank $n$ lattice embedding of Lemma~\ref{l:lisca} we have the following procedure: 
\begin{enumerate}
\item Let $E_2 \subset E$ be the set of basis vectors $e$ such that $e$ is a component vector of some vector $v$ of weight $-2$.
\item Let $E^k_2$ be the subset of $E_2$ consisting of keystone vectors.
A basis vector $e \in E$ is a {\em keystone} for the embedding of $V$ if there is a filtration $E = E^0 \supset E^1 \supset \dots \supset E^n = \emptyset$ such that $E^{j-1} - E^{j} = \{e^j\}$, $e^1 = e$, and for $2\leq j \leq n$ there exists a vector $v^j\in V$ whose embedding projects to $\pm e^j$ in the lattice spanned by $E^{j-1}$.
\item Given $e_i \in E^k_2$ let $(v_j,\epsilon_j)$ be the set of vectors $v_j$ with linking number $\epsilon_j = e_i \cdot v_j\neq 0$.
\item Find an oriented unknot $K$ that links component $j$ of the chain link with linking number $\epsilon_j$ and is invariant with respect to the involution $u$.
\item Check that $-1$ surgery on $K$ produces $S^1 \times S^2$.
\end{enumerate}

We say such a $K$ obtained by the above manner is ``suggested by the lattice embedding''.   Note that if $e \in E_2$ is a keystone for the embedding of $V$, then we may view $K$ as giving an embedded vector $\phi(v_K)=\pm e$.

\begin{proof}[Lattice Embedding Proof of Theorem~\ref{thm:main}]
We need to show that for each lens space in the theorem there is at least one knot with a longitudinal surgery yielding $S^{1}\times S^{2}$. From Lemma~\ref{l:lisca} we know that the lens spaces we need to consider coincide with the lens spaces obtained by surgery on the black diagrams in Figure~\ref{fig:eqsurgeries}. 

Applying the above described heuristic to the types (1)--(7) we find the red and green knots in Figure~\ref{fig:eqsurgeries} finishing the proof. We sketch how the procedure runs for the various types.

For each type (2), (3), (5), (6), (7) the set $E_2$ is fixed and $E_2^k$ is easily determined. Moreover, notice that if a keystone vector appears in a $-2$-chain then all basis vectors appearing in the $-2$-chain are keystone vectors. Each of these keystones yields a different unknot in the fourth step of the above procedure. However, all the unknots thus obtained are related to one another by the handle slides described in Figure~\ref{fig:handleslide} and therefore yield the same linking types. The embedding suggests two different linking types yielding $S^{1}\times S^{2}$, in red and green in Figure~\ref{fig:eqsurgeries}, for types (2) and (3) and only one, in red, for types (5), (6), (7).

For types (1) and (4), the set $E_2$ depends on sequences $b_1, \dots, b_k$ and $c_1, \dots, c_\ell$.  Recall from the end of section~\ref{sec:liscasembeddings} that these two types, and hence these two sequences, are generated from the ``seed'' strings $(-2,-2,-2)$ and $(-3,-2,-2,-3)$ respectively by applications of the two operations (a) and (b).  The corresponding lattice embeddings of these seeds and their expansions by the operations (a) and (b) are also indicated.   
For these embeddings of the seeds, the set $E_2^k$ is easily determined.   Upon expansions by the operations (a) and (b), the set $E_2^k$ for type (1) is seen to partition according to a central $-2$--chain and a $-2$--chain at either end, while for type (4) $E_2^k$ remains associated to the central $-2$--chain.   For each of these partitions we find a knot suggested by the embedding shown in red or green in Figure~\ref{fig:eqsurgeries} types (1) and (4).
\end{proof}

\begin{lemma}\label{lem:BGembedding}
The lens space duals to the {\sc bg} and {\sc gofk} knots are the knots suggested by the lattice embeddings. 
\end{lemma}

\begin{proof}
By Lemma~\ref{lem:uniqueembedding} the lattice embeddings for types (1)--(7) above are essentially unique. The different knots suggested by the embedding are then, as explained in the preceding proof, precisely the red and green curves in Figure~\ref{fig:eqsurgeries}. In section~\ref{s:correspondence} we show that these knots do correspond to the duals to the {\sc bg} and {\sc gofk} knots: using Kirby calculus we relate the diagrams in Figure~\ref{fig:eqsurgeries} with the knots in Figure~\ref{fig:chainsurgery}.
\end{proof}

\begin{remark}
The dotted blue unknot in type (3) of Figure~\ref{fig:eqsurgeries} corresponds to the family of {\sc spor} knots when the parameter $t=1$, as shown in section~\ref{s:correspondence}. These are {\em not} suggested by the embedding as we now demonstrate.
\end{remark}

\begin{figure}[h]
\centering
\executeiffilenewer{expanded2.svg}{expanded2.eps}%
{inkscape -z -D --file=expanded2.svg %
--export-eps=expanded2.eps --export-latex}%
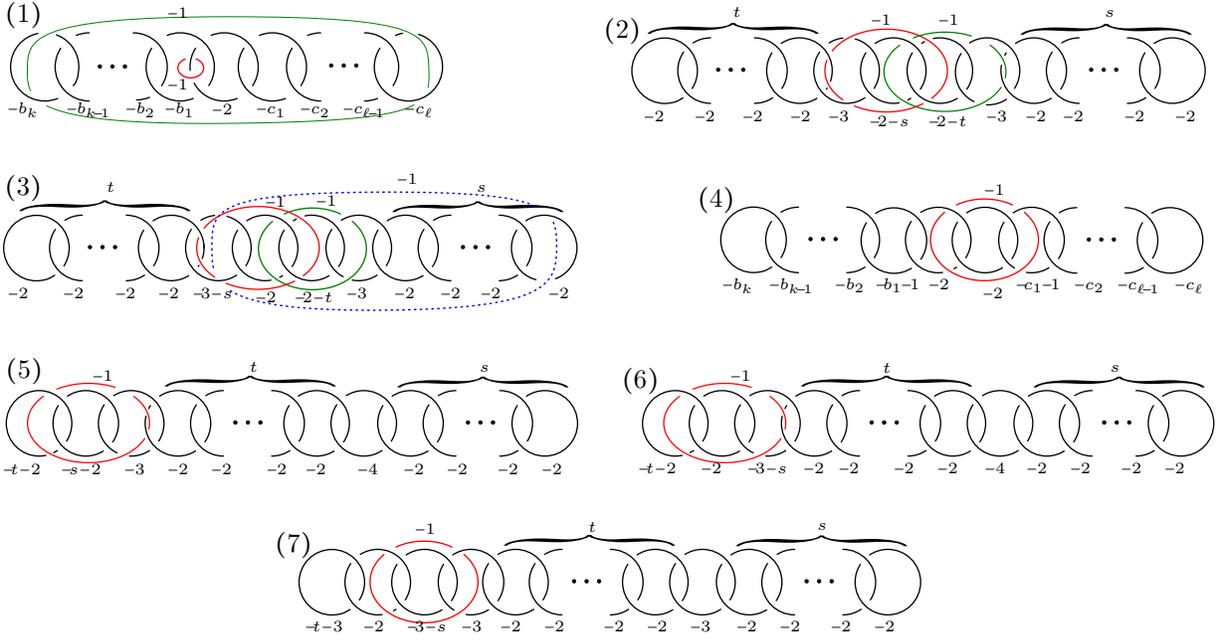%

\caption{Equivariant surgeries yielding $S^1 \times S^2$}
\label{fig:eqsurgeries}
\end{figure}

\begin{figure}[h]
\centering
\includegraphics[height=1in]{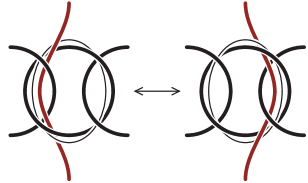}
\caption{A handle slide of a red curve across a $-2$--framed component of a chain link.  The lighter curve indicates the framing.}
\label{fig:handleslide}
\end{figure}

\subsection{Explicit example: type (3), $t=1$, $s=n-2$}\label{sec:explicitexample}
In order to understand the duals of the {\sc spor} knots in the lens spaces we are going to work in full detail the case $L(p_n,q_n)=L(16n^2-8n+1,16n-2)$ where $\frac{p_n}{q_n} = \frac{16n^2-8n+1}{16n-2}=[2, n+1, 2, 3, 3, 2^{[n-2]}]^-$.
The string of surgery coefficients for the chain link, $(-2,-n-1,-2,-3,-3,-2^{[n-2]})$, shows that all the lens spaces in this example are of type (3) of Lemma~\ref{l:lisca} with $t=1$ and $s=n-2$.  From Lisca's embeddings above we thus obtain the following explicit embedding of the intersection lattice $(\Z^{n+3},Q_{p_n,q_n})$ into the standard negative diagonal lattice of rank $n+3$:

\[
\begin{array}{r|rrrrrrrrrrrr}
&e_1& e_2&e_3&e_4&e_5&e_6&e_7&e_8& \cdots & e_{n+2} & e_{n+3}\\
\hline
v_1 &     &   & -&  &+  & &&  &  \\
v_2 &    & + &+ &+  &  & + & + & + & \cdots  & +& + \\
v_3 &   + & - & &  &  &  &  &  &  \\
v_4 &    & + & -&  &-  &  &  &  &  \\
v_5 &   - & - & & + &  &  &  &  &  \\
v_6 &    &  & & - &  & + &  &  &  \\
v_7 &    &  & &  &  & - & + &  &  \\
v_8       &  &  & &  &  &  &-  &+  &  \\
      \vdots&   &  & &  &  &  &  &  & \ddots \\
v_{n+3}  &&  &  & &  &  &  &  &  & - &+ \\
\end{array}
\]

Here we see that $E_2 = E = \{ e_1, \dots, e_{n+3}\}$ and one may easily check that $E_2^k = \{e_1, e_2, e_3, e_5\}$ and therefore the embedding suggests $4$ possible unknots. However, the reader may check that the unknots obtained from $e_{1}$ and $e_{2}$ are related by a handle slide and so are the two unknots defined by $e_{3}$ and $e_{5}$. A Kirby calculus argument shows that $-1$ surgery on either of these two unknots embedded in $L(p_n,q_n)$ yields $S^{1}\times S^{2}$. The unknot related to $e_{1}$ is shown embedded in $L(p_n,q_n)$ as the red $1$-framed curve on the top right diagram of Figure\ref{fig:quotients} and the one defined by $e_{3}$ corresponds to the red $-1$-framed curve on the top left diagram of Figure\ref{fig:quotients}. Comparing the quotients with Figure~\ref{fig:BGIIIandV} we obtain that the unknots suggested by the embedding correspond to the duals of {\sc bgiii} and {\sc bgv}.

In order to understand how the family {\sc spor} in Figure~\ref{fig:sporadic} relates to the embedding we turn the surgery on the chain link with coefficients $(-2,-n-1,-2,-3,-3,-2^{[n-2]})$ into a surgery on a $6$ component chain link with coefficients $(-2,-n,2,-1,2,n)$ and lift the {\sc spor} bands to $L(p_n,q_n)$. As a result we obtain that the {\sc spor} bands lift to the red $0$--framed curve in the bottom diagram of Figure~\ref{fig:quotients}. This red curve corresponds to an unknot $K$ with framing $-1$ that satisfies $\phi(v_{K})=\pm e_{n+3}$. This basis vector is not a keystone, since blowing down the unknot $K$ does not prompt a sequence of blow downs yielding a single $0$-framed unknot. In fact, after blowing down $K$ we are led to a surgery diagram with coefficients $\{-2,-2,-2,-3,-2\}$ on a chain link in which the rightmost unknot is linked to the second. The corresponding embedding in the negative standard lattice of rank $5$ is

{\small
\[
\begin{array}{r|c@{\hspace{0.5mm}}c@{\hspace{0.5mm}}c@{\hspace{0.5mm}}c@{\hspace{0.5mm}}c}
    &e_1&e_2&e_3&e_4&e_5\\
\hline
v'_1     &  &    & - &   & + \\
v'_2     &  & + & + &    &\\
v'_3     & + & - &   &  &\\
v'_4     &  & + & - &   & -\\
v'_5     & - & - &   &  & 
\end{array}
\]}%
and the $3$--manifold described by this surgery is $S^{1}\times S^{2}$.

\begin{figure}
\centering
\executeiffilenewer{quotients.svg}{quotients.eps}%
{inkscape -z -D --file=quotients.svg %
--export-eps=quotients.eps --export-latex}%
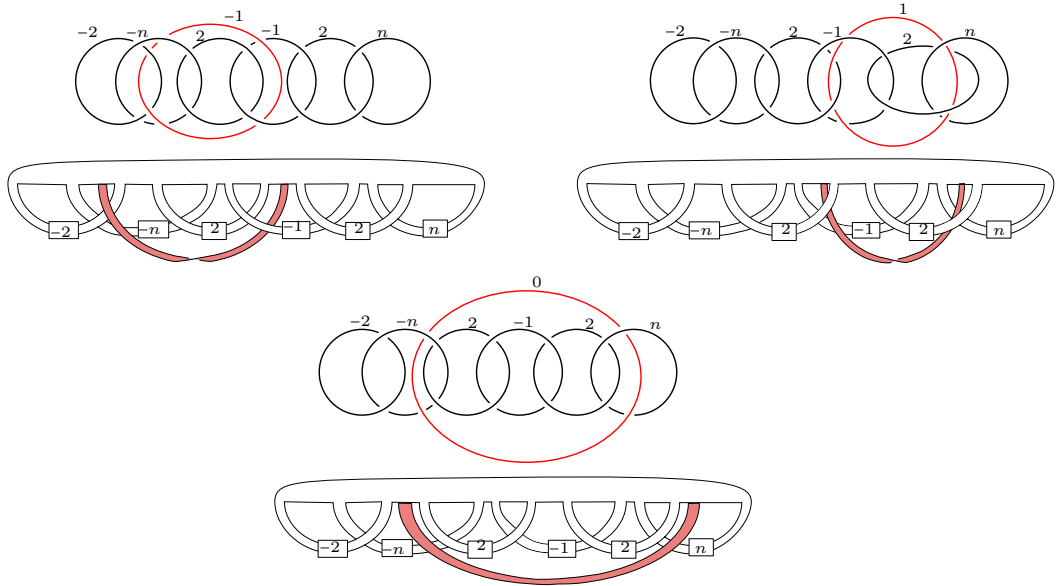%

\caption{Three copies of a chain link representing $L(16n^2-8n+1,16n-2)$ with its two-bridge quotient $K(16n^2-8n+1,16n-2)$ below are shown.  For each chain link an additional framed red curve is shown representing a knot in a lens space with an $S^1 \times S^2$ surgery; its corresponding banding of the two-bridge link is also shown.  
The red curves in the top diagrams represent the knots suggested by the embedding for the family of lens spaces $L(16n^2-8n+1,16n-2)$ which are the duals of the families {\sc bgiii} and {\sc bgv}.  In the bottom diagram the red curve represents the dual of the {\sc spor} band which we obtained by lifting to the lens space the corresponding band in the two bridge link $K(16n^2-8n+1,16n-2)$.}
\label{fig:quotients}
\end{figure}

Notice that the two first surgeries, the ones corresponding to $e_1$ and $e_3$, turn the more general family of two bridge links with projection $L(-2^{[k]},-n-1,-2,-2+k,-3,-2^{[n-2]})$ into $S^1\times S^2$. However, the last surgery described, which corresponds to the {\sc{spor}} knots, does not describe such a surgery in this more general family.


\section{Kirby calculus and tangle quotients}

\subsection{Correspondence between Figure~\ref{fig:eqsurgeries} and Figure~\ref{fig:chainsurgery}}\label{s:correspondence}

Using Kirby calculus we identify the colored curves in Figure~\ref{fig:eqsurgeries} with the corresponding duals to the {\sc bg}, the {\sc gokf} and the {\sc spor} knots shown in Figure~\ref{fig:chainsurgery}.
We will detail the transition from the first four types in Figure~\ref{fig:eqsurgeries} to Figure~\ref{fig:chainsurgery} and sketch the remaining.

\medskip

{\bf Types (1) and (4):}
Let $L(p,q)$, with $\frac{p}{q}=[a_1,\dots,a_n]^-$ and $a_i\geq 2$, be the boundary of the $4$--dimensional plumbing manifold $P(p,q)$ given by the Kirby diagram in Figure~\ref{fig:general}. Blowing up several times one of the final unknots we can turn its framing to $-1$ and after blowing down this $-1$--framed unknot it is not difficult to see that we can continue blowing up and down all the curves in the diagram until we obtain a new Kirby diagram having only positive framings $\geq 2$ and whose boundary is still $L(p,q)$. The $\ell$--tuple of positive framings $(a'_1,\dots,a'_\ell)$ is related to $(-a_1,\dots,-a_n)$ by an easy algorithm known as Riemenscheneider point rule \cite{pointrule}.    We have that $-L(p,q)=L(p,p-q)$ and a Kirby diagram for $-L(p,q)$ can be obtained from that of $L(p,q)$ by changing the sings of all the framings. It follows that the framings in the negative diagram associated to $L(p,q)$ and the framings in the \emph{negative} diagram associated to $-L(p,q)$ are related to one another by Riemenschneider point rule. Notice that $\frac{q}{p}+\frac{p-q}{p}=1$, which is the condition defining the framings $(-c_1,\dots,-c_\ell)$ from $(-b_1,\dots,-b_k)$ for types $(1)$ and $(4)$ in Lemma~\ref{l:lisca}. Therefore, a series of blow ups and blow downs changes the framings in types $(1)$ and $(4)$ as follows:
\begin{itemize}
\item[$(1)$] $(-b_k,-b_{k-1},\dots,-b_1,-2,-c_1,\dots,-c_{l-1},-c_\ell)\longrightarrow (-b_k,-b_{k-1},\dots,-b_1,-1,b_1,\dots,b_{k-1},b_k)$
\item[$(4)$] $(-b_k,-b_{k-1},\dots,-b_1-1,-2,-2,-1-c_1,\dots,-c_{l-1},-c_\ell)\longrightarrow (-b_k,-b_{k-1},\dots,-b_1-1,-2,-1,2,b_1,\dots,b_{k-1},b_k).$
\end{itemize}
In order to obtain the correspondence between Figure~\ref{fig:eqsurgeries} and Figure~\ref{fig:chainsurgery} for types (1) and (4), we need to perform the Kirby moves that change the negative plumbing associated to $-L(p,q)$ into the positive one, while taking into account how do the colored curves in Figure~\ref{fig:eqsurgeries} change. 

\smallskip
The case of type $(4)$ is worked out in full detail in Figure~\ref{fig:neg-standBGII}. The second diagram is obtained from type $(4)$ in Figure~\ref{fig:eqsurgeries} by blowing up the clasp between the $-2$ and the $-c_1-1$ framed unknots. Several blow downs and isotopies lead to the fifth diagram in Figure~\ref{fig:neg-standBGII}. At this point we start performing the above described suite of blowing ups and downs to turn the framings $-c_i$ to positive integers $\geq 2$; the first step corresponds to the sixth diagram. The third to last diagram corresponds to the mirror image of the second diagram in Figure~\ref{fig:chainsurgery} with opposite framings. This means that the corresponding lens spaces have the opposite orientation. It follows that type $(4)$ in Figure~\ref{fig:eqsurgeries} corresponds to the mirror image of {\sc bgii} understood as in Figure~\ref{fig:chainsurgery}.

\begin{figure}
\centering
\executeiffilenewer{11to13BGII.svg}{11to13BGII.eps}%
{inkscape -z -D --file=11to13BGII.svg %
--export-eps=11to13BGII.eps --export-latex}%
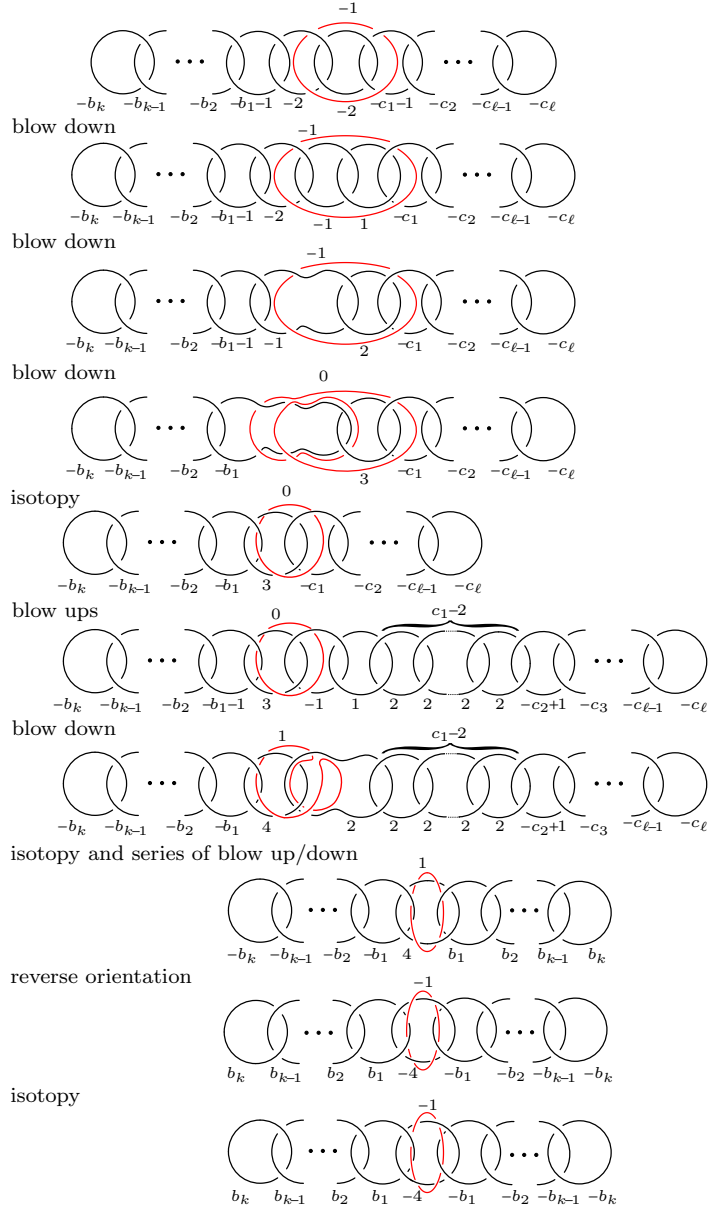%

\caption{Kirby calculus from Type (4), Figure~\ref{fig:eqsurgeries} to {\sc bgii}, Figure~\ref{fig:chainsurgery}.}
\label{fig:neg-standBGII}
\end{figure}

\smallskip
The case of type (1) is very similar to that of type (4). On the one hand, if we consider the lens space diagram, in black, with the red curve, it is clear that when changing the $-c_i$ to positive integers the red curve is unaffected. In this way we obtain that the red curve in type $(1)$ corresponds to {\sc bgi} in Figure~\ref{fig:chainsurgery}. 

The case of type $(1)$ with the green curve is slightly more delicate. The process of turning the $-c_i$ framings to positive integers yields the second diagram in Figure~\ref{fig:neg-standGOFK}. Notice that the framing of the green curve changes to $0$ and that the clasps between each two consecutive unknots are opposite in both sides of the central $-1$. The black diagram cannot be isotoped into a chain of unknots with only right claps without twisting the green curve: in the process the last unknot gains $k$ half twists as illustrated in the last diagram of Figure~\ref{fig:neg-standGOFK}. This diagram shows that the green curve in Family $1$ of Figure~\ref{fig:eqsurgeries} corresponds to the {\sc gokf} curve in Figure~\ref{fig:chainsurgery}.  Note that in Figure~\ref{fig:chainsurgery} there are $2k$ unknots at each side of the $-1$, this can always be achieved by blowing up once one of the clasps in the last diagram of Figure~\ref{fig:neg-standGOFK}. In Figure~\ref{fig:chainsurgery} the $2k$ half twists are taken into account by means of the $\frac{1}{k}$ framed curve.

\begin{figure}
\centering
\executeiffilenewer{11to13GOFK.svg}{11to13GOFK.eps}%
{inkscape -z -D --file=11to13GOFK.svg %
--export-eps=11to13GOFK.eps --export-latex}%
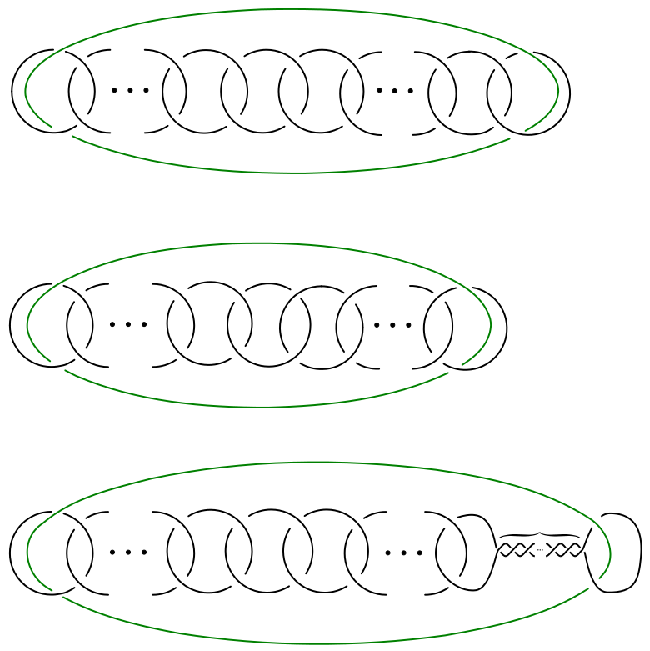%

\caption{Kirby calculus from the green knot in Type (1), Figure~\ref{fig:eqsurgeries} to {\sc gofk}, Figure~\ref{fig:chainsurgery}}
\label{fig:neg-standGOFK}
\end{figure}

\medskip
{\bf Type (2):}
The transformation of type $(2)$ in Figure~\ref{fig:eqsurgeries} is detailed in Figure~\ref{fig:neg-standBGIV}. We start changing the two $-2$--chains into single unknots with framings $t$ and $s$. At this point we blow down the two black $-1$--framed curves, changing the framings of the red and green curves, and obtaining the last diagram in Figure~\ref{fig:neg-standBGIV} which shows that the red and green curve in type $(2)$ Figure~\ref{fig:eqsurgeries} correspond to the {\sc bgiv} in Figure~\ref{fig:chainsurgery} with $t,s\geq 0$. 

\begin{figure}
\centering
\executeiffilenewer{From11to13family2.svg}{From11to13family2.eps}%
{inkscape -z -D --file=From11to13family2.svg %
--export-eps=From11to13family2.eps --export-latex}%
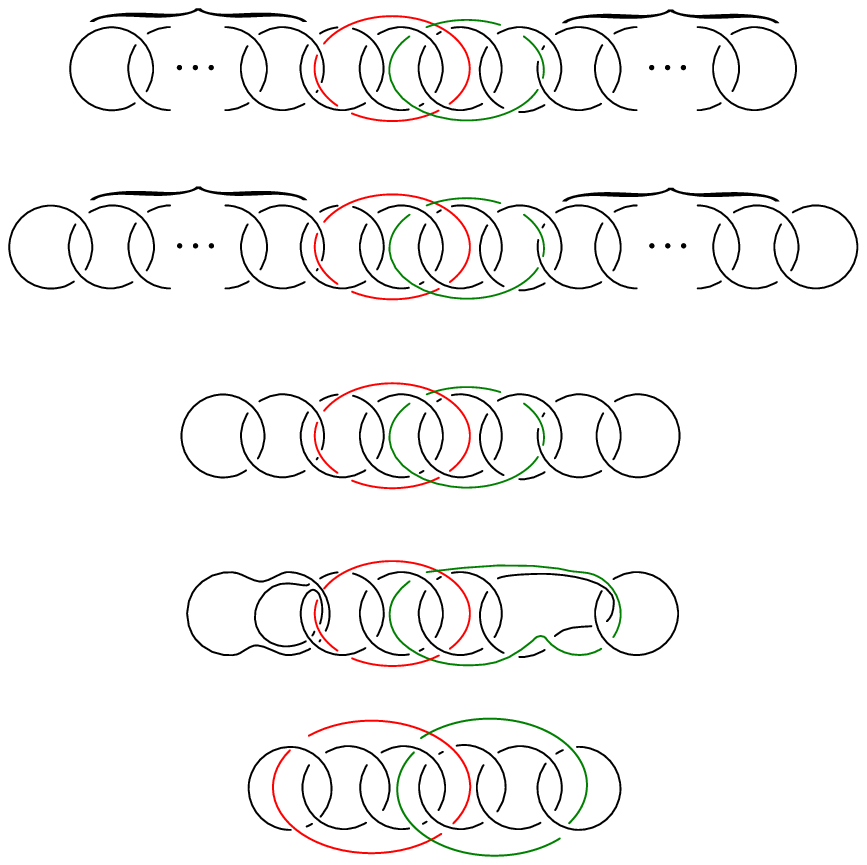%

\caption{Kirby calculus from Type (2), Figure~\ref{fig:eqsurgeries} to {\sc bgiv}, Figure~\ref{fig:chainsurgery}.}
\label{fig:neg-standBGIV}
\end{figure}

\medskip
{\bf Type (3):}
The green curve in type $(3)$ in Figure~\ref{fig:eqsurgeries} coincides with the {\sc bgv} diagram in Figure~\ref{fig:chainsurgery}. In order to see this we start by changing the two final $-2$--chains in type $(3)$ into two unknots of framings $t+1$ and $s+1$ (just like for type $(2)$). Since the green curve is unaffected by this change in the black diagram it follows that it corresponds to {\sc bgv} in Figure~\ref{fig:chainsurgery} with $s,t\geq 0$. 

The case of the red curve in type $(3)$ is very similar. This time, the last blow down changing the $-2$--chain with $t$ unknots, changes the framing of the red curve to $0$. It is then easy to see that it corresponds to {\sc bgiii} in Figure~\ref{fig:chainsurgery} with $s,t\geq 0$. 

Finally, the case of the blue curve in type $(3)$, with a surgery changing the lens space into $S^1\times S^2$ when $t=1$, is considered in Figure~\ref{fig:neg-standSPOR}. We start blowing up the clasp between the $-3$--framed unknot and the leftmost unknot in the $-2$--chain and changing the framing of the first unknot from $-2$ to $2$. After several blow downs we arrive to the third diagram of Figure~\ref{fig:neg-standSPOR}. The final blow down of the $-1$--framed unknot changes the framing of the blue curve to $0$ showing that it corresponds to the {\sc spor} curve in Figure~\ref{fig:chainsurgery} with $s,t\geq 0$.

\begin{figure}
\centering
\executeiffilenewer{11to13SPOR.svg}{11to13SPOR.eps}%
{inkscape -z -D --file=11to13SPOR.svg %
--export-eps=11to13SPOR.eps --export-latex}%
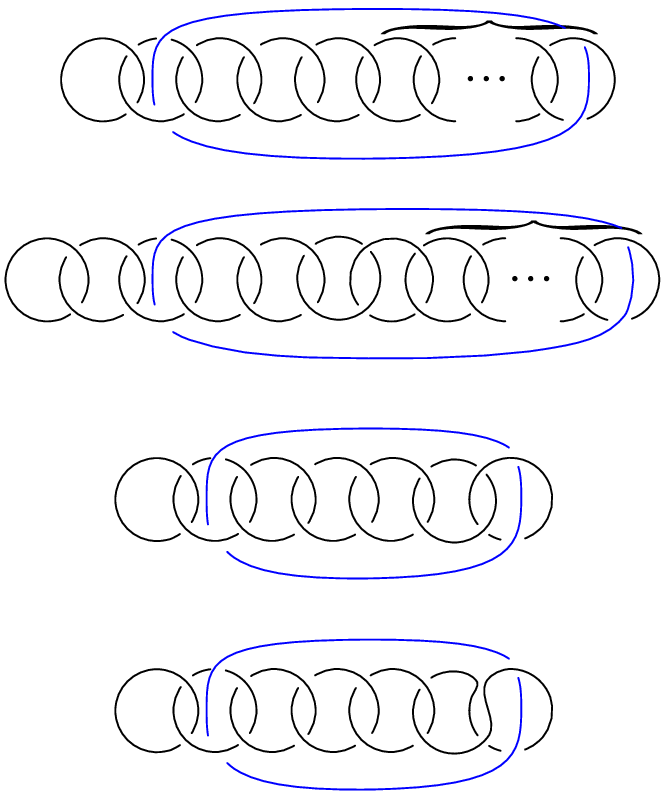%

\caption{Kirby calculus from blue knot in Type (3) with $t=1$, Figure~\ref{fig:eqsurgeries} to {\sc spor}, Figure~\ref{fig:chainsurgery}}
\label{fig:neg-standSPOR}
\end{figure}

\medskip
{\bf Types (5), (6), (7):}
Type $(5)$ in Figure~\ref{fig:eqsurgeries} is the same as {\sc bgiii} in Figure~\ref{fig:chainsurgery}. The correspondence is established changing the two $-2$--chains into two unknots with framings $t+1$ and $s+1$. This yields a chain of $6$ unknots with framings $(-t-2,-s-2,-2,t+1,-2,s+1)$ which coincides with the framings in the third diagram in Figure~\ref{fig:chainsurgery} after rescaling the parameter $-t-2$ to $t+1$. In the process the red curve's framing will change to $0$ (cf.\ type $(2)$) yielding {\sc bgiii} in Figure~\ref{fig:chainsurgery} with $s\geq 0$ and $t\leq -3$.  

By the same argument, type $(6)$ in Figure~\ref{fig:eqsurgeries} corresponds to the red {\sc bgiv} from Figure~\ref{fig:chainsurgery} with $s\geq 0$ and $t\leq -3$ (after rescaling of the parameter $t$ to $-t-3$).  

Finally, the transformation of type $(7)$ starts again by changing the $-2$--chains into two unknots with framings $t+1$ and $s+1$. This does not change the red curve. We obtain a chain of six unknots with framings $(-t-3,-2,-3-s,-2,t+1,-1,s+1)$ that rescaled with $t$ replaced by $-s-4$ and $s$ by $t-1$ coincides with the {\sc bgv} diagram in Figure~\ref{fig:chainsurgery} for $s\leq -4$ and $t\geq 1$.

\subsection{From chain link surgery descriptions to tangle descriptions}\label{sec:chaintotangle}
Figures~\ref{fig:BGIchaintotangle}, \ref{fig:BGIIchaintotangle}
, \ref{fig:GOFKchaintotangle}, \ref{fig:BGIIIandBGVandSPORchaintotangle}, and \ref{fig:BGIVchaintotangle}  explicitly show how the quotient of these knots with $S^1 \times S^2$ surgeries by the involution $u$ correspond to the bandings of Figures~\ref{fig:BGIandII}, \ref{fig:BGIIIandV}, \ref{fig:BGIVandIV2}, \ref{fig:twobandings}, and \ref{fig:sporadic}.

 \begin{figure}
\centering
\includegraphics[width=4in]{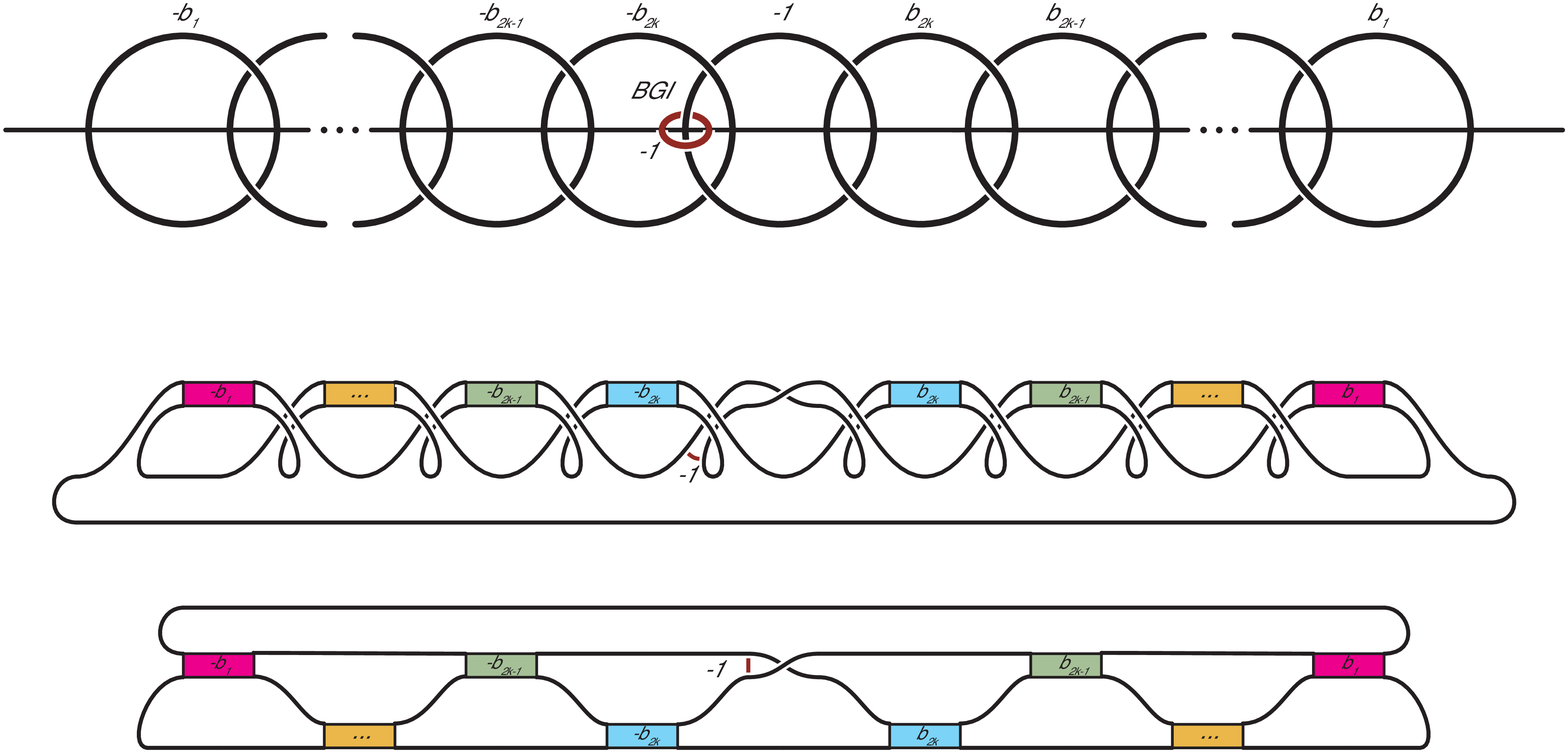}
\caption{}
\label{fig:BGIchaintotangle}
\end{figure}

 \begin{figure}
\centering
\includegraphics[width=4in]{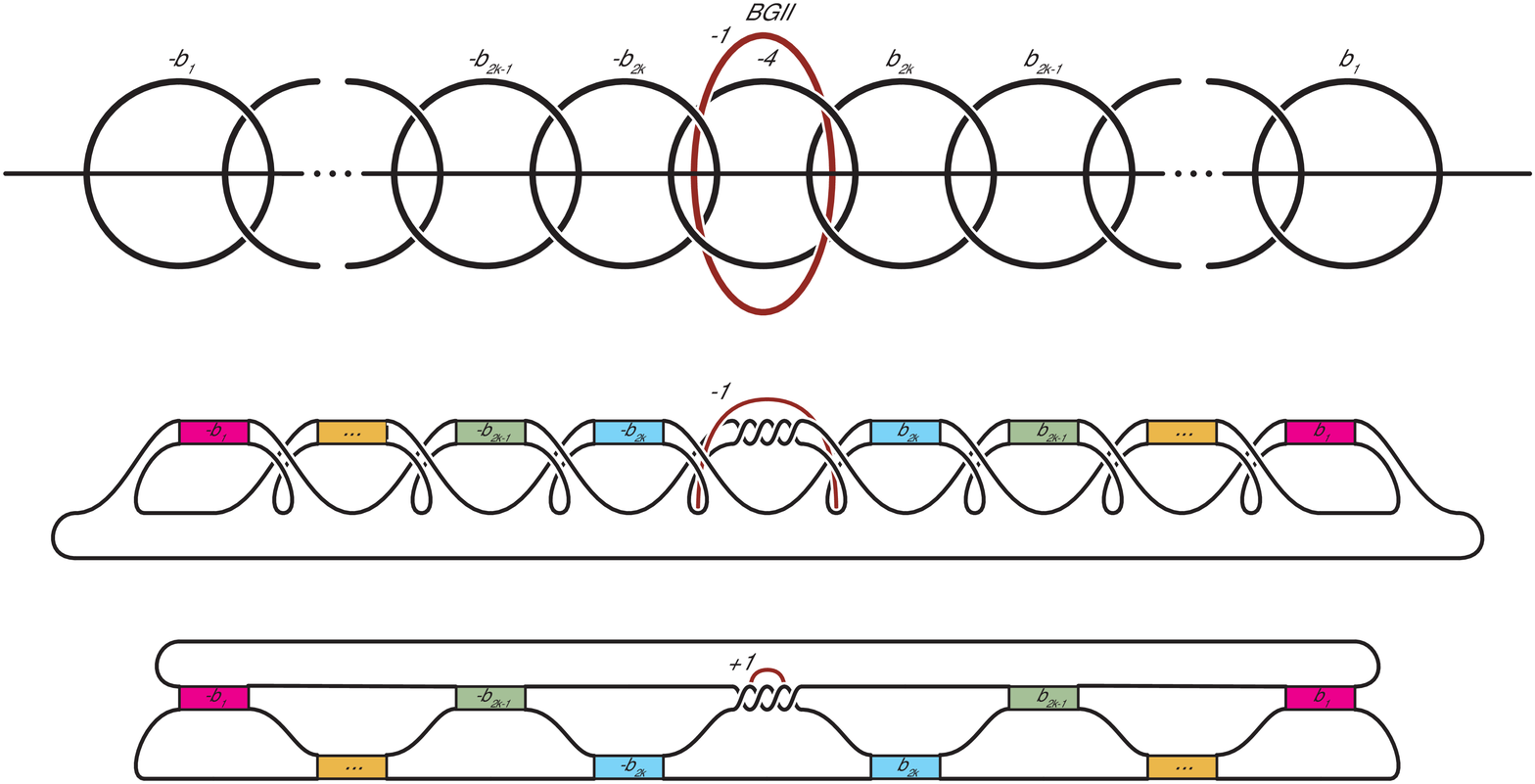}
\caption{}
\label{fig:BGIIchaintotangle}
\end{figure}

 \begin{figure}
\centering
\includegraphics[width=4in]{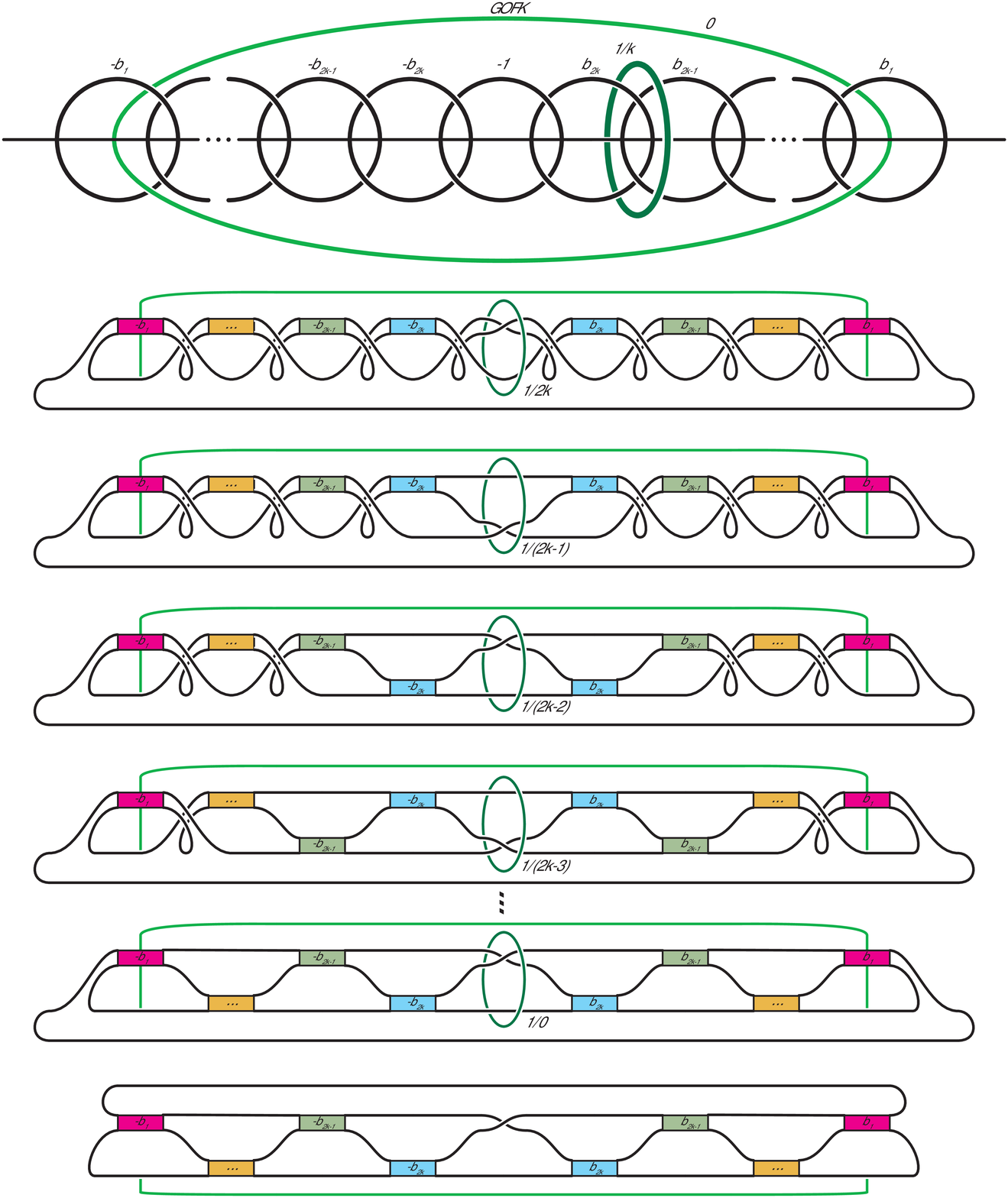}
\caption{}
\label{fig:GOFKchaintotangle}
\end{figure}

 \begin{figure}
\centering
\includegraphics[width=5in]{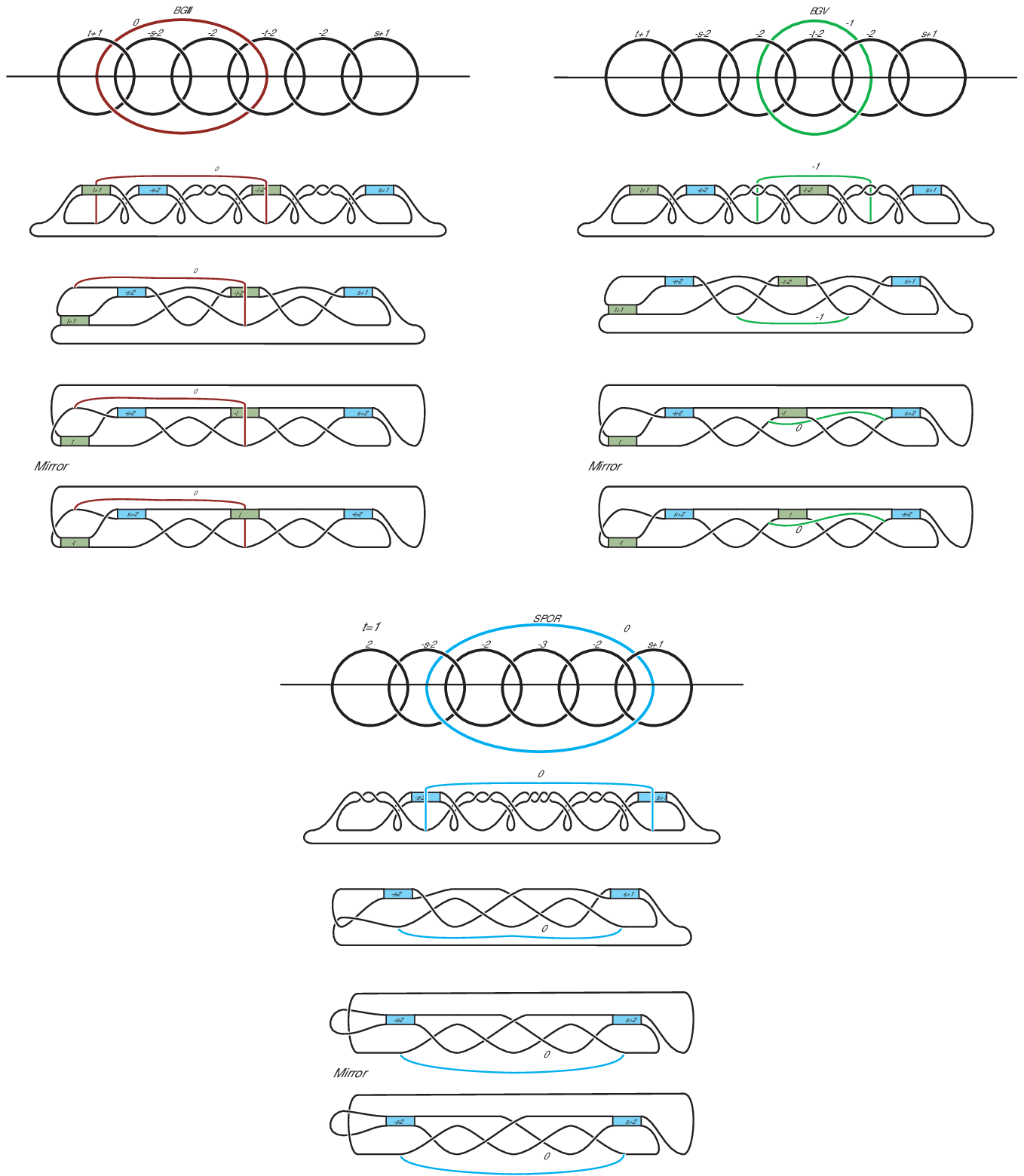}
\caption{}
\label{fig:BGIIIandBGVandSPORchaintotangle}
\end{figure}

 \begin{figure}
\centering
\includegraphics[width=5in]{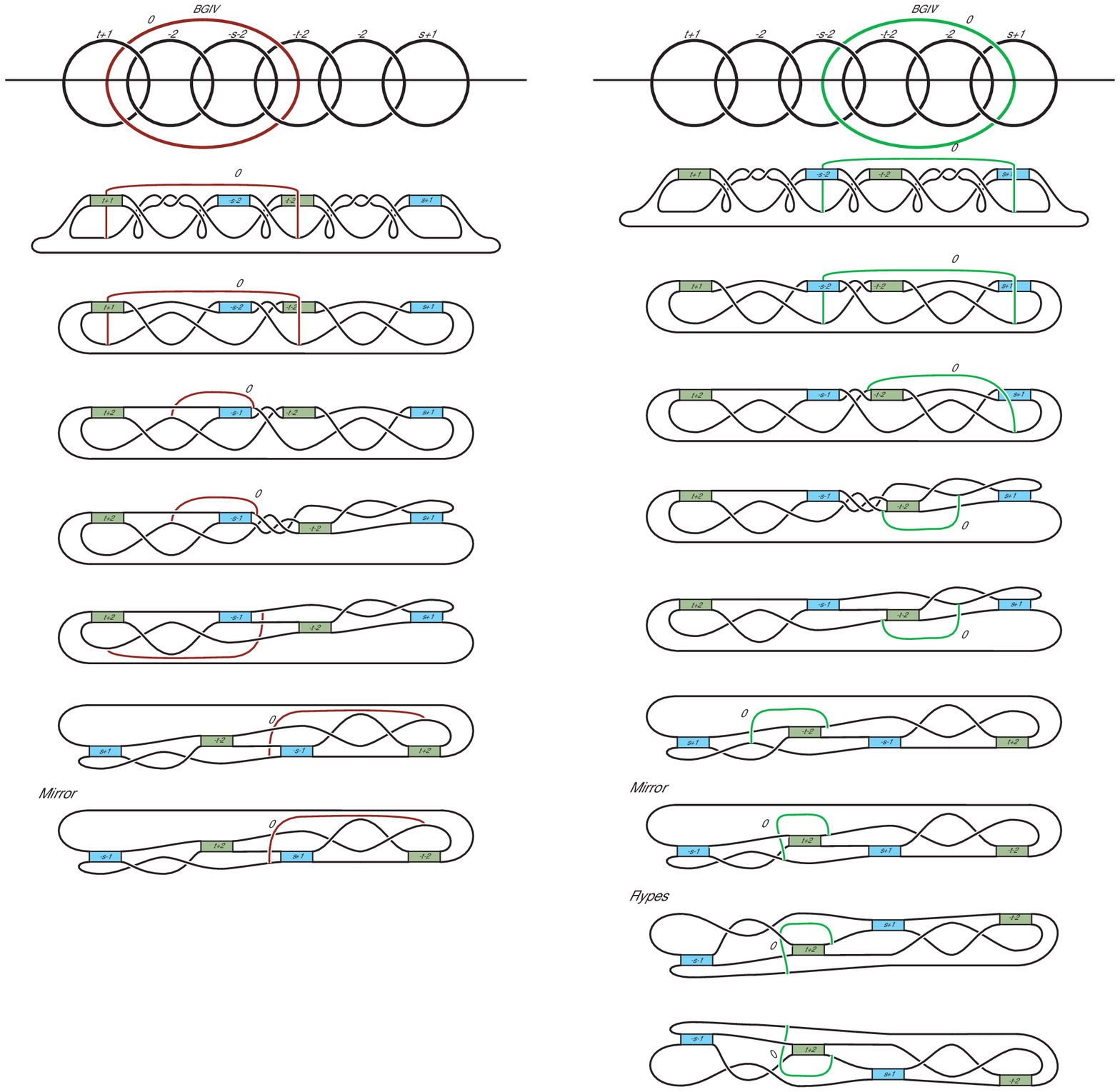}
\caption{}
\label{fig:BGIVchaintotangle}
\end{figure}

\section{Spherical braids: A proof of Theorem~\ref{thm:sphericalbraid}.}\label{sec:braids}

\begin{proof}[Proof of Theorem~\ref{thm:sphericalbraid}]
Any identification of a solid torus containing a braid with a Heegaard solid torus of $S^1 \times S^2$, produces a spherical braid.  Hence the Berge-Gabai knots are spherical braids.   (Note that any two such identifications of a solid torus with a Heegaard torus are related by isotopy within $S^1 \times S^2$, mirroring, and inverting the $S^1$ direction.  Since Berge-Gabai knots in solid tori are invariant under inverting the $S^1$ direction, each Berge-Gabai knot in a solid torus gives a unique Berge-Gabai knot in $S^1 \times S^2$ up to mirroring.)

View the genus one fibered knot in $S^1 \times S^2$ as a Hopf band plumbed onto an annulus with trivial monodromy as displayed on the right-hand side of Figure~\ref{fig:hopfjump}. Figure~\ref{fig:isotopiesofGOFcurves} shows the effect of a particular isotopy of the once-punctured torus fiber, realizing the monodromy, upon the core of the annulus (red) and the core of the Hopf band (blue).  While most of the isotopy of Figure~\ref{fig:isotopiesofGOFcurves} occurs near the original fiber, the last stage of the isotopy exploits that the ambient manifold is $S^1 \times S^2$ much like the ``lightbulb trick'' and is highlighted in Figure~\ref{fig:hopfjump}.  Using this isotopy or its inverse, one may arrange any curve that lies on the fiber to run along the train track  shown on the left-hand side of Figure~\ref{fig:gofkbraid}.  (See, for example, \cite{penner-harer} for the fundamentals of train tracks.)  The right-hand side of Figure~\ref{fig:gofkbraid} shows an isotopy of the fiber and the train track so that any curve carried by the train track is a spherical braid.  Thus every \gofk\ knot is a spherical braid.

\begin{figure}
\centering
\includegraphics[height=2in]{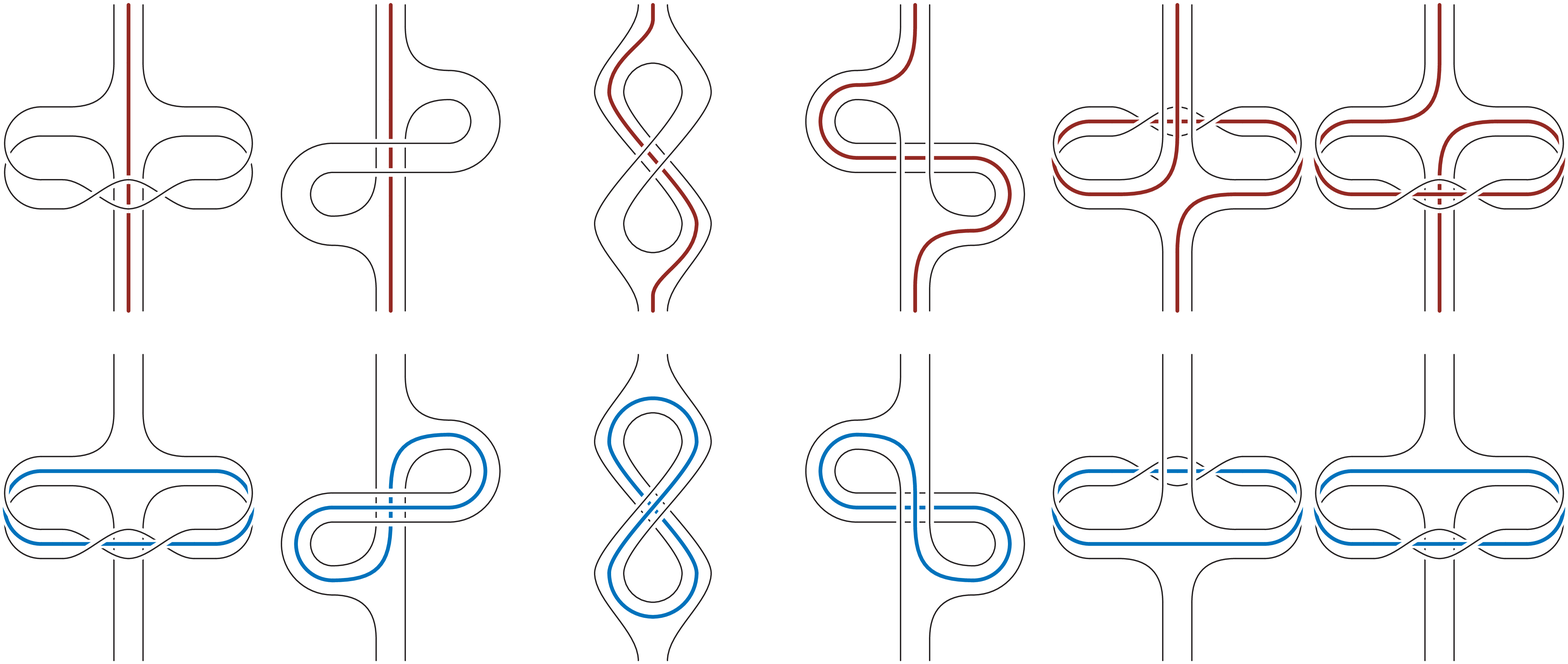}
\caption{}
\label{fig:isotopiesofGOFcurves}
\end{figure}

\begin{figure}
\centering
\includegraphics[height=1in]{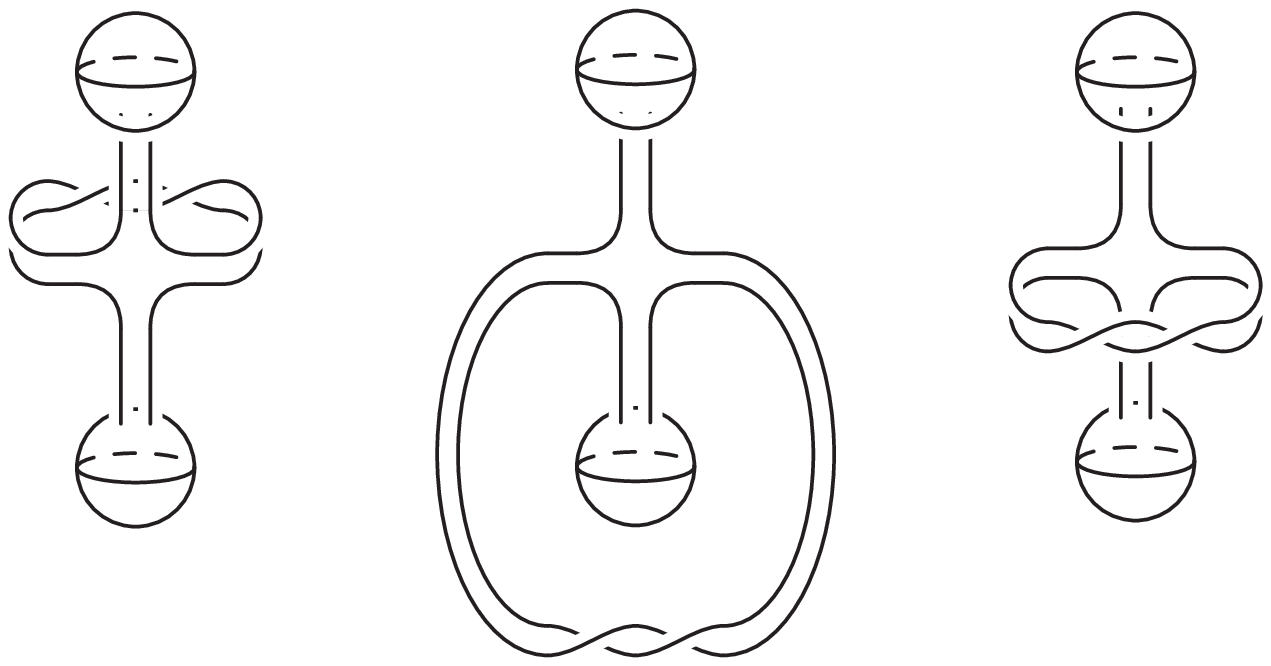}
\caption{}
\label{fig:hopfjump}
\end{figure}

\begin{figure}
\centering
\includegraphics[height=1in]{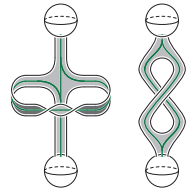}
\caption{By an isotopy of the once-punctured torus containing this train track into a more ``flattened'' form, curves carried by the train track are seen to be spherical braids.}
\label{fig:gofkbraid}
\end{figure}

\begin{figure}
\centering
\includegraphics[width=5.5in]{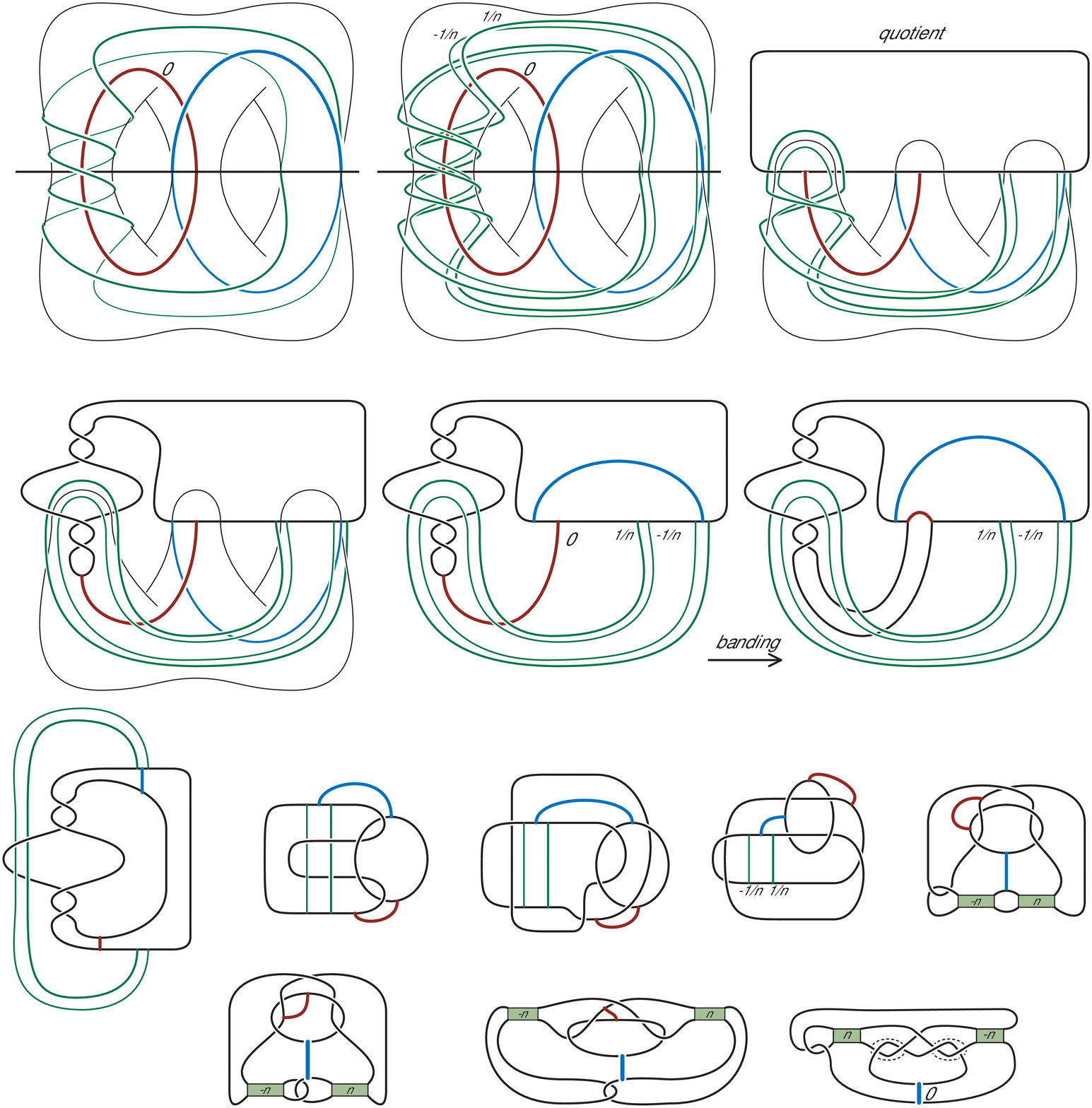}
\caption{}
\label{fig:sporDPtoTangle}
\end{figure}

The bottom right picture in Figure~\ref{fig:sporDPtoTangle} is our tangle version of the sporadic knots in $S^1 \times S^2$ also shown in the second picture of Figure~\ref{fig:sporadic}.  The top left picture of Figure~\ref{fig:sporDPtoTangle} gives a doubly primitive presentation of these sporadic knots, in blue, in terms of Dehn twists along the green curve and $0$--surgery on the red curve as described in the caption.  In this picture one may observe that right handed Dehn twists along the green curve keeps the blue curve braided about the red curve.  One may also check that after an isotopy, left handed Dehn twists will braid the blue curve about the red too.   Hence after $0$--surgery, the resulting blue curve is a spherical braid in $S^1 \times S^2$.
\end{proof}

\bibliographystyle{amsplain}
\bibliography{twobridge}

\end{document}